\documentclass{article}
\usepackage{amsmath}
\usepackage[T1]{fontenc}
\usepackage[latin9]{luainputenc}
\usepackage{geometry}
\usepackage{color}
\usepackage{amsmath}
\usepackage{amsthm}
\usepackage{amssymb}
\usepackage{graphicx}
\usepackage[unicode=true,
 bookmarks=true,bookmarksnumbered=true,bookmarksopen=false,
 breaklinks=false,pdfborder={0 0 1},backref=false,colorlinks=true]
 {hyperref}
\usepackage{subfig}
\usepackage{authblk}
\hypersetup{
 linkcolor=blue, citecolor=black, urlcolor=blue, filecolor=blue, pdfpagelayout=OneColumn, pdfnewwindow=true, pdfstartview=XYZ, plainpages=false}
\makeatletter
\numberwithin{equation}{section}
  \theoremstyle{remark}
  \newtheorem*{rem*}{\protect\remarkname}
\theoremstyle{plain}
\newtheorem{thm}{\protect\theoremname}[section]
  \theoremstyle{definition}
  \newtheorem{defn}[thm]{\protect\definitionname}
  \theoremstyle{definition}
  \newtheorem{example}[thm]{\protect\examplename}
  \theoremstyle{plain}
  \newtheorem{lem}[thm]{\protect\lemmaname}
  \theoremstyle{plain}
  
  \theoremstyle{remark}
  \newtheorem{rem}[thm]{\protect\remarkname}
  \theoremstyle{plain}
  \newtheorem{cor}[thm]{\protect\corollaryname}

\usepackage{ifpdf} 
\ifpdf 

 \IfFileExists{lmodern.sty}{\usepackage{lmodern}}{}

\fi 

\usepackage[figure]{hypcap}

\newcommand{\R}{{\mathbb R}}

\usepackage{fancyhdr}

\@ifundefined{extrarowheight}
 {\usepackage{array}}{}
\date{ }

\def\definitionname{Definition}
\def\theoremname{Theorem}
\def\propositionname{Proposition}
\def\lemmaname{Lemma}
\def\corollaryname{Corollary}
\def\remarkname{Remark}
\def\examplename{Example}

\makeatother

  \providecommand{\corollaryname}{Corollary}
  \providecommand{\definitionname}{Definition}
  \providecommand{\examplename}{Example}
  \providecommand{\lemmaname}{Lemma}
  \providecommand{\propositionname}{Proposition}
  \providecommand{\remarkname}{Remark}
\providecommand{\theoremname}{Theorem}
\title{Hamilton-Jacobi equations for optimal control on networks with entry or exit costs}
\author{Manh Khang DAO\textsuperscript{}\thanks{\noindent\textsuperscript{}IRMAR, Universit\'e de Rennes 1, 35000 Rennes, France. manh-khang.dao@univ-rennes1.fr}}
\providecommand{\keywords}[1]{\textbf{Key words:} #1}
\providecommand{\AMSclass}[1]{\textbf{AMS subject classification:} #1}
\begin{document}

\date{\today}




\maketitle
\begin{abstract}
We consider an optimal control on networks in the spirit of the
works of Achdou et al. (2013) and Imbert et al. (2013). The main
new feature is that there are entry (or exit) costs at the edges
of the network leading to a possible discontinuous value function.
We characterize the value function as the unique viscosity solution
of a new Hamilton-Jacobi system. The uniqueness is a consequence
of a comparison principle for which we give two different proofs,
one with arguments from the theory of optimal control inspired by Achdou et al. (2014) and one based
on partial differential equations techniques inspired by a recent
work of Lions and Souganidis (2016). 
\end{abstract}
\\
\AMSclass{34H05, 35F21, 49L25, 49J15, 49L20, 93C30}
\section{Introduction}

A network (or a graph) is a set of items, referred to as vertices
or nodes, which are connected by edges (see Figure~\ref{fig: Network intro}
for example). Recently, several research projects have been devoted to
dynamical systems and differential equations on networks,  in general or more particularly 
in connection with problems of data transmission or traffic management (see for example Garavello and Piccoli~\cite{GP2006}
and Engel et al~\cite{EKNS2008}). 

An optimal control problem is an optimization problem where an agent tries
to minimize a cost which depends on the solution of a controlled ordinary
differential equation (ODE). The ODE is controlled in the sense that
it depends on a function called the control. The goal is to find the
best control in order to minimize the given cost. In many situations, the optimal value 
of the problem as a function of the initial state (and possibly of the initial time when the horizon of the problem is finite)
is a viscosity solution of a Hamilton-Jacobi-Bellman partial differential equation 
(HJB equation). Under appropriate conditions, the HJB equation has a unique viscosity
solution characterizing by this way the value function. Moreover, the optimal control may be recovered from the solution of the HJB equation, at least if the latter is smooth enough.

The first articles  about optimal control problems in which the set of admissible states 
is a network (therefore the state variable is a continuous one) appeared in 2012: 
in~\cite{ACCT2013}, Achdou et al. derived the  HJB equation associated to an infinite horizon
 optimal control on a network and proposed a suitable notion of viscosity solution. Obviously, 
 the main difficulties arise at the vertices where the network does not have a regular differential structure.
  As a result, the new admissible test-functions whose restriction to each edge is
$C^{1}$ are applied.  Independently and at the same time, Imbert et al.~\cite{IMZ2013}
proposed an equivalent notion of viscosity solution for studying a
Hamilton-Jacobi approach to junction problems and traffic flows. Both~\cite{ACCT2013} and~\cite{IMZ2013} contain first results
 on comparison principles which were improved later. It is also worth mentioning 
 the work by Schieborn and Camilli~\cite{SC2013}, in which
the authors focus on eikonal equations on networks and on a less general
notion of viscosity solution. In the particular case of eikonal equations, 
Camilli and Marchi established in~\cite{CM2013} the equivalence between 
the definitions given in~\cite{ACCT2013,IMZ2013,SC2013}.

Since 2012, several proofs of comparison principles for HJB equations on networks, giving uniqueness of the solution, have been 
proposed.
\begin{enumerate}
\item  In~\cite{AOT2015}, Achdou et al. give a proof of a comparison principle for 
a stationary HJB equation arising from an optimal control with infinite horizon, (therefore the Hamiltonian is convex)
by mixing arguments from the theory of optimal control and PDE techniques.
Such a proof was much inspired by works of Barles et al.~\cite{BBC2014,BBC2013}, on
 regional  optimal control problems in $\R^d$, (with discontinuous dynamics and costs).
\item A different and more general proof, using only arguments from the theory of PDEs was obtained by Imbert and Monneau
in~\cite{IM2017}. The proof works for quasi-convex Hamiltonians, and for stationary and time-dependent 
HJB equations. It relies on the construction of suitable~\emph{vertex test functions}. 
\item A very simple and elegant proof, working for non convex Hamiltonians, has been very recently given 
by  Lions and Souganidis~\cite{LS2016,LS2017}.
\end{enumerate}
The goal of this  paper is to consider an optimal control problem
on a network in which there are entry (or exit)
 costs at each edge of the network and to study the related HJB equations.
The effect of the entry/exit costs is to make the value function of the problem 
discontinuous.  Discontinuous solutions of Hamilton-Jacobi
equation have been studied by various authors, see for example Barles~\cite{Barles1993}, Frankowska and Mazzola~\cite{FM2013},
and in particular Graber et al.~\cite{GHZ2017} for different HJB equations on networks with discontinuous solutions. 

 To simplify the problem, we will first study
the case of junction, i.e., a network  of the form $\mathcal{G}= \cup_{i=1}^N \Gamma_{i}$
with  $N$ edges $\Gamma_{i}$ ($\Gamma_{i}$ is the  closed half line $\R^+ e_i$) and 
only one vertex $O$, where $\{O\}= \cap_{i=1}^N \Gamma_{i}$.
 Later, we will  generalize our analysis to  networks with an arbitrary number of vertices.
 In the case of the  junction described above, our assumptions about the dynamics and the running costs
 are similar to those made in~\cite{AOT2015}, except that additional 
 costs $c_{i}$ for entering  the edge $\Gamma_{i}$ at $O$
or  $d_{i}$ for exiting $\Gamma_{i}$ at $O$ are added  in the cost
functional. Accordingly, the value function is continuous on $\mathcal{G}$,
but is in general discontinuous at the vertex $O$. Hence, instead
of considering the value function $\mathsf{v}$, we split it into the  
 collection $(v_i)_{1\le i\le N}$, where $v_i$ is  continuous function defined on the edge
 $\Gamma_{i}$.
More precisely,
\[
v_{i}\left(x\right)=\begin{cases}
\mathsf{v}\left(x\right) & \text{if }x\in \Gamma_{i}\backslash\left\{ O\right\} ,\\
\lim_{\delta \rightarrow 0 ^+}\mathsf{v}\left(\delta e_{i}\right) & \text{if }x=O.
\end{cases}
\]
Our approach is therefore reminiscent of optimal switching problems (impulsional control):
 in the present case the switches can only occur at the vertex $O$.
Note that our assumptions will ensure that 
$\mathsf{v}|_{\Gamma_i\setminus \{O\}}$ is Lipschitz continuous near $O$ and that
  $\lim_{\delta \rightarrow 0 ^+}\mathsf{v}\left(\delta e_{i}\right)$ does exist. 
In the case of entry costs for example, our first main result will be to find 
the relation between
 $\mathsf{v}\left(O\right)$, $v_{i}\left(O\right)$ and
$v_{j}\left(O\right)+c_{j}$ for $i,j=\overline{1,N} $. 

 This will show that the functions $(v_i)_{1\le i\le N}$ are (suitably defined) viscosity solutions 
of the following system
\begin{equation}
\begin{array}{cc}
\lambda u_{i}\left(x\right)+H_{i}\left(x,\dfrac{d u_{i}}{d x_{i}}\left(x\right)\right)=0 & \mbox{if \ensuremath{x\in \Gamma_{i}\backslash \left\{O\right\}},}\\
{\displaystyle \lambda u_{i}\left(O\right)+\max\left\{ -\lambda\min_{j\ne i}\left\{ u_{j}\left(O\right)+c_{j}\right\} ,H_{i}^{+}\left(O,\dfrac{d u_{i}}{dx_{i}}\left(O\right)\right),H_{O}^{T}\right\} =0} & \mbox{if \ensuremath{x=O}}.
\end{array}\label{eq: Hamilton-Jacobi equation}
\end{equation}
Here $H_i$ is the Hamiltonian corresponding to edge $\Gamma_{i}$. At vertex $O$, 
the definition of the Hamiltonian has to be
particular, in order to consider all the possibilities when $x$ is close
to $O$. More specifically, if $x$ is close to $O$ and belongs to
$\Gamma_{i}$ then: 
\begin{itemize}
\item  The term $\min_{j\ne i}\left\{ u_{j}\left(O\right)+c_{j}\right\} $
accounts for situations in which the trajectory enters $\Gamma_{i_{0}}$ where
$u_{i_{0}}\left(O\right)+c_{i_{0}}=\min_{j\ne i}\left\{ u_{j}\left(O\right)+c_{j}\right\} $.
\item  The term $H_{i}^{+}\left(O,\dfrac{d u_{i}}{dx_{i}}\left(O\right)\right)$
accounts for situations in which the trajectory does not leave $\Gamma_{i}$.
\item  The term $H_{O}^{T}$ accounts for  situations in which  the
trajectory stays at $O$. 
\end{itemize}

The most important part of the paper will be devoted
to two different proofs of a comparison principle leading to the well-poseness of~\eqref{eq: Hamilton-Jacobi equation}: the first one uses arguments from optimal
control theory coming from Barles et al.~\cite{BBC2013,BBC2014}
and  Achdou et al.~\cite{AOT2015}; the second one is inspired by Lions and Souganidis~\cite{LS2016} and uses arguments from the 
theory of PDEs.

The paper is organized as follows: Section~\ref{sec:optim-contr-probl} deals with the optimal
control problems with entry and exit costs: we give a simple example in which the value function is discontinuous at the vertex $O$,
and also prove results on the structure of the value function near $O$. 
In Section~\ref{sec:hamilt-jacobi-syst}, the new system of~\eqref{eq: Hamilton-Jacobi equation} is defined 
and a suitable notion of  viscosity solutions is proposed. In Section~\ref{sec:relation-between-two}, 
we prove our value functions are viscosity solutions of the above mentioned system.
 In Section~\ref{sec:comp-princ-uniq}, some properties of viscosity sub and super-solution
are given and  used to obtain the comparison principle. Finally, optimal control problems
with  entry costs which may be zero and related HJB
equations  are considered  in Section~\ref{sec:gener-case-switch}.

\begin{figure}
\includegraphics[scale=0.5]{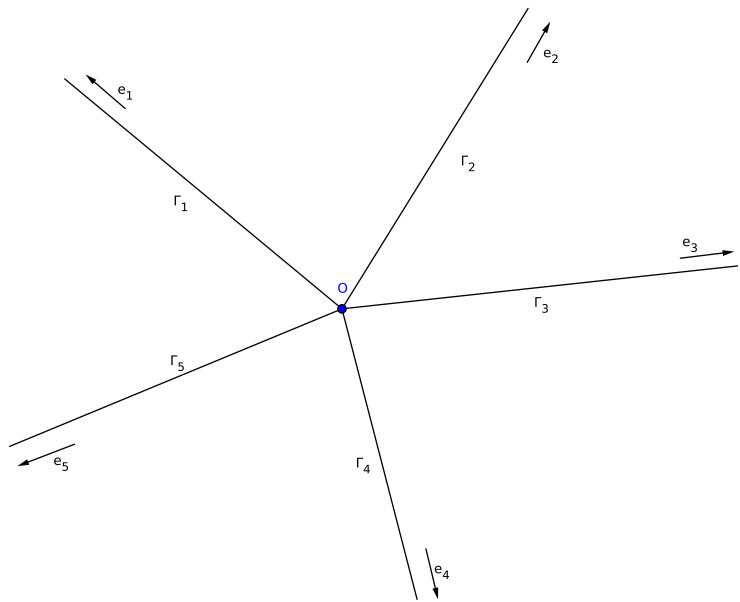}\caption{\label{fig: Network intro}The network $\mathcal{G}$
($N=5$)}

\end{figure}

\section{Optimal control problem on junction with entry/exit costs}\label{sec:optim-contr-probl}

\subsection{The geometry}

We consider the model case of the junction in $\mathbb{R}^{d}$
with $N$ semi-infinite straight edges, $N>1$. The edges are denoted
by $\left(\Gamma_{i}\right)_{i=\overline{1,N}}$ where $\Gamma_{i}$ is the closed
half-line $\mathbb{R}^{+}e_{i}$. The vectors $e_{i}$ are two by
two distinct unit vectors in $\mathbb{R}^{d}$.  The half-lines $\Gamma_{i}$
are glued at the vertex $O$ to form the junction $\mathcal{G}$

\[
\mathcal{G}=\bigcup_{i=1}^{N}\Gamma_{i}.
\]
The geodetic distance $d\left(x,y\right)$ between two points $x,y$
of $\mathcal{G}$ is
\[
d\left(x,y\right)=\begin{cases}
\left|x-y\right| & \mbox{if }x,y\mbox{ belong to the same egde }\Gamma_{i},\\
\left|x\right|+\left|y\right| & \mbox{if }x,y\mbox{ belong to different edges }\Gamma_{i}\mbox{ and }\Gamma_{j}.
\end{cases}
\]

\subsection{The optimal control problem}

We consider infinite horizon optimal control problems which have different
dynamic and running costs for each and every edge. For $i=\overline{1,N}$,
\begin{itemize}
\item the set of control on $\Gamma_{i}$ is denoted by $A_{i}$
\item  the system
is driven by a dynamics $f_{i}$
\item there is a running cost $\ell_{i}$.
\end{itemize}
Our main assumptions, referred to as $\left[H\right]$ hereafter,  are as follows:
\begin{description}
\item{$\left[H0\right]$ (\textbf{Control sets})} 
Let $A$ be a metric space (one can take $A=\mathbb{R}^d$. For $i=\overline{1,N}$, $A_{i}$ is  a nonempty compact subset of $A$ and the sets $A_{i}$ are disjoint.
\item{$\left[H1\right]$ (\textbf{Dynamics})}
For $i=\overline{1,N}$, the function $f_{i}:\Gamma_{i}\times A_{i}\rightarrow\mathbb{R}$
is continuous and bounded by $M$. Moreover, there exists $L>0$ such that 
\[
\left|f_{i}\left(x,a\right)-f_{i}\left(y,a\right)\right|\le L\left|x-y\right|\quad\mbox{ for all } x,y\in \Gamma_{i}, a\in A_{i}.
\]
Hereafter, we will use the notation $F_{i}\left(x\right)$ for the
set $\left\{ f_{i}\left(x,a\right)e_{i}:a\in A_{i}\right\} $.
\item{$\left[H2\right]$ (\textbf{Running costs})} 
For $i=\overline{1,N}$, the function $\ell_{i}:\Gamma_{i}\times A_{i}\rightarrow\mathbb{R}$ is
a continuous function bounded by $M>0$. There exists a modulus of
continuity $\omega$ such that
\[
\left|\ell_{i}\left(x,a\right)-\ell_{i}\left(y,a\right)\right|\le\omega\left(\left|x-y\right|\right)\quad\mbox{ for all } x,y\in \Gamma_{i},a\in A_{i}.
\]
\item{$\left[H3\right]$ (\textbf{Convexity of dynamic and costs})} For $x\in \Gamma_{i}$, the following
set
\[
\mbox{FL}_{i}\left(x\right)=\left\{ \left(f_{i}\left(x,a\right)e_{i},\ell_{i}\left(x,a\right)\right):a\in A_{i}\right\} 
\]
 is non-empty, closed and convex.
\item{$\left[H4\right]$ (\textbf{Strong controllability})} There exists a real
number $\delta>0$ such that
\[
\left[-\delta e_{i},\delta e_{i}\right]\subset F_{i}\left(O\right)=\left\{ f_{i}\left(O,a\right)e_{i}:a\in A_{i}\right\} .
\]
\end{description}

\begin{rem}
The assumption that the  sets $A_{i}$ are disjoint
is not restrictive. Indeed, if $A_{i}$ are not disjoint, then we
define $\tilde{A}_{i}=A_{i}\times\left\{ i\right\} $ and $\tilde{f}_{i}\left(x,\tilde{a}\right)=f_{i}\left(x,a\right),\tilde{\ell}_{i}\left(x,\tilde{a}\right)=\ell_{i}\left(x,a\right)$
 with $\tilde{a}=\left(a,i\right)$ with $a\in A_{i}$. The assumption
$\left[H3\right]$ is made to avoid the use of relaxed control. With
assumption $\left[H4\right]$, one gets that the Hamiltonian which
will appear later is coercive for $x$ close to the $O$. Moreover, $\left[H4\right]$ is an important assumption to prove Lemma~\ref{Control t at x near O} and Lemma~\ref{lem:test function at O}.

\end{rem}
Let
\[
\mathcal{M}=\left\{ \left(x,a\right):x\in\mathcal{G}, a\in A_{i}\mbox{ if }x\in \Gamma_{i}\backslash\left\{ O\right\} ,\mbox{ and }a\in\cup_{i=1}^{N}A_{i}\mbox{ if }x=O\right\} .
\]
Then $\mathcal{M}$ is closed. We also define the function on $\mathcal{M}$
by
\[
\mbox{for all}\left(x,a\right)\in\mathcal{M},\quad f\left(x,a\right)=\begin{cases}
f_{i}\left(x,a\right)e_{i} & \quad\mbox{if }x\in \Gamma_{i}\backslash\left\{ O\right\} \mbox{ and }a\in A_{i},\\
f_{i}\left(O,a\right)e_{i} & \quad\mbox{if }x=O\mbox{ and }a\in A_{i}.
\end{cases}
\]
The function $f$ is continuous on $\mathcal{M}$ since the sets $A_{i}$
are disjoint.
\begin{defn}[\emph{The speed set and the admissible control set}]
 The set $\tilde{F}\left(x\right)$ which contains all the ``possible
speeds'' at $x$ is defined by
\[
\tilde{F}\left(x\right)=\begin{cases}
F_{i}\left(x\right) & \quad\mbox{if }x\in \Gamma_{i}\backslash\left(O\right),\\
\bigcup_{i=1}^{N}F_{i}\left(O\right) & \quad\mbox{if \ensuremath{x=O}}.
\end{cases}
\]
For $x\in\mathcal{G}$, the set of admissible trajectories starting
from $x$ is
\[
Y_{x}=\left\{ y_{x}\in Lip\left(\mathbb{R}^{+};\mathcal{G}\right):\begin{cases}
\dot{y}_{x}\left(t\right) & \in\tilde{F}\left(y_{x}\left(t\right)\right) \quad\mbox{for a.e. }t>0 \\
y_{x}\left(0\right) & =x
\end{cases}\right\} .
\]
\end{defn}
According to \cite[Theorem 1.2]{AOT2015}, a solution $y_{x}$
can be associated with several control laws. We introduce the set
of admissible controlled trajectories starting from $x$
\[
\mathcal{T}_{x}=\left\{ \left(y_{x},\alpha\right)\in L_{loc}^{\infty}\left(\mathbb{R}^{+};\mathcal{M}\right):y_{x}\in Lip\left(\mathbb{R}^{+};\mathcal{G}\right)\mbox{ and }y_{x}\left(t\right)=
x+\int_{0}^{t}f\left(y_{x}\left(s\right),\alpha\left(s\right)\right)ds\right\} .
\]
Notice that, if $\left(y_{x},\alpha\right)\in\mathcal{T}_{x}$ then $y_{x}\in Y_{x}$. Hereafter, we will denote $y_{x}$ by $y_{x,\alpha}$ if $\left(y_{x},\alpha\right)\in\mathcal{T}_{x}$. For any $y_{x,\alpha}$, we can define the closed set $T_O=\left\{ t\in\mathbb{R}^{+}:y_{x,\alpha}\left(t\right)=O\right\} $
and the open set $T_{i}$  in $\mathbb{R}^{+}=\left[0,+\infty\right)$
 by $T_i=\left\{ t\in\mathbb{R}^{+}:y_{x,\alpha}\left(t\right)\in \Gamma_{i}\backslash\left\{ O\right\} \right\} $.
The set $T_{i}$ is a countable union of disjoint open intervals

\[
T_{i}=\bigcup_{k\in K_{i}\subset\mathbb{N}}T_{ik}=\begin{cases}
\left[0,\eta_{i0}\right)\cup\bigcup_{k\in K_{i}\subset\mathbb{N^{\star}}}\left(t_{ik},\eta_{ik}\right) & \quad\mbox{if }x\in \Gamma_{i}\backslash\left\{ O\right\} ,\\
\bigcup_{k\in K_{i}\subset\mathbb{N^{\star}}}\left(t_{ik},\eta_{ik}\right) & \quad\mbox{if }x\notin \Gamma_{i}\backslash\left\{ O\right\} ,
\end{cases}
\]
where 
$K_{i}=\overline{1,n}$ if the trajectory $y_{x,\alpha}$ enters $\Gamma_{i}$ $n$ times 
and $K_{i}=\mathbb{N}$ if the trajectory $y_{x,\alpha}$ enters $\Gamma_{i}$ infinite times.
\begin{rem}
From the above definition, one can see that $t_{ik}$ is an entry time
in $\Gamma_{i}\backslash\left\{ O\right\}$ and $\eta_{ik}$
is an exit time from $\Gamma_{i}\backslash\left\{ O\right\} $
. Hence
\[
y_{x,\alpha}\left(t_{ik}\right)=y_{x,\alpha}\left(\eta_{ik}\right)=O.
\]
\end{rem}
Let $C=\left\{ c_{1},c_{2},\ldots,c_{N}\right\} $ be a set of \textbf{\emph{entry costs}} 
and $D=\left\{ d_{1},d_{2},\ldots,d{}_{N}\right\} $ be a set of \textbf{\emph{exit costs}}.  We underline that, except in  Section~\ref{sec:gener-case-switch}, entry and exist costs are positive.

In the sequel, we define two different \emph{cost functionals} (the first one corresponds to the case when there is a cost for entering the edges
and the second one  corresponds to the case when there is a cost for exiting the edges):
\begin{defn}[\textbf{The cost functionals and value functions with entry/exit costs}]
The costs associated to trajectory $\left(y_{x,\alpha},\alpha\right)\in\mathcal{T}_{x}$ are
defined by
\[
J\left(x;\left(y_{x,\alpha},\alpha\right)\right)=\int_{0}^{+\infty}\ell\left(y_{x,\alpha}
\left(t\right),\alpha\left(t\right)\right)e^{-\lambda t}dt+\sum_{i=1}^{N}
\sum_{k\in K_{i}}c_{i}e^{-\lambda t_{ik}}\quad\mbox{(cost functional with entry cost)},
\]
and
\[
\widehat{J}\left(x;\left(y_{x,\alpha},\alpha\right)\right)=\int_{0}^{+\infty}\ell\left(y_{x,\alpha}\left(t\right),\alpha\left(t\right)\right)e^{-\lambda t}dt+\sum_{i=1}^{N}\sum_{k\in K_{i}}d_{i}e^{-\lambda\eta_{ik}}\quad\mbox{(cost functional with exit cost)},
\]
where the running cost $\ell:\mathcal{M}\rightarrow\mathbb{R}$ is
\[
\ell\left(x,a\right)=\begin{cases}
\ell_{i}\left(x,a\right) & \quad\mbox{if \ensuremath{x\in \Gamma_{i}\backslash\left\{ O\right\} }}\mbox{ and }a\in A_{i},\\
\ell_{i}\left(O,a\right) & \quad\mbox{if }x=0\mbox{ and }a\in A_{i}.
\end{cases}
\]
Hereafter, to simplify the notation, we will use  $J\left(x,\alpha \right)$ and $\widehat{J}\left(x,\alpha \right)$ instead of $J\left(x;\left(y_{x,\alpha},\alpha\right)\right)$ and $\widehat{J}\left(x;\left(y_{x,\alpha},\alpha\right)\right)$, respectively.

The value functions of the infinite
horizon optimal control problem  are  defined by:
\[
\mathsf{v}\left(x\right)=\inf_{\left(y_{x,\alpha},\alpha\right)\in\mathcal{T}_{x}}J\left(x;\left(y_{x,\alpha},\alpha\right)\right)\quad\mbox{(value function with entry cost)},
\]
and
\[
\widehat{\mathsf{v}}\left(x\right)=\inf_{\left(y_{x,\alpha},\alpha\right)\in\mathcal{T}_{x}}\widehat{J}\left(x;\left(y_{x,\alpha},\alpha\right)\right)\quad\mbox{(value function with exit cost)}.
\]
\end{defn}
\begin{rem}
By the definition of the value function, we are mainly interested in a
control law $\alpha$ such that $J\left(x,\alpha\right)<+\infty$.
In such a  case, if $\left|K_{i}\right|=+\infty$, then we can order
$\left\{ t_{ik},\eta_{ik}: \;k\in\mathbb{N} \right\}$
such that
\[
t_{i1}<\eta_{i1}<t_{i2}<\eta_{i2}<\ldots<t_{ik}<\eta_{ik}<\ldots,
\]
and
\[
\lim_{k\rightarrow\infty}t_{ik}=\lim_{k\rightarrow\infty}\eta_{ik}=+\infty.
\]
Indeed, assuming if  $\lim_{k\rightarrow\infty}t_{ik}=\overline{t}<+\infty$, then
\begin{eqnarray*}
J\left(x,\alpha\right) & \ge & -\dfrac{M}{\lambda}+\sum_{k=1}^{+\infty}e^{-\lambda t_{ik}}c_{i}=-\dfrac{M}{\lambda}+c_{i}\sum_{k=1}^{+\infty}e^{-\lambda t_{ik}}=+\infty,
\end{eqnarray*}
in contradiction with $J\left(x,\alpha\right)<+\infty$. This means that the state cannot switch edges infinitely many times in finite time, otherwise the cost functional is obviously infinite.
\end{rem}
The following example shows that the value function with entry
costs is possibly discontinuous (The same holds for the value function with exit costs).
\begin{example}
Consider the network $\mathcal{G}=\Gamma_{1}\cup \Gamma_{2}$ where $\Gamma_{1}=\mathbb{R}^{+} e_{1}=\left(-\infty,0\right]$
and $\Gamma_{2}=\mathbb{R}^{+}e_{2}=\left[0,+\infty\right)$. The control
sets are $A_{i}=\left[-1,1\right]\times\left\{ i\right\} $ with $i\in\left\{1,2\right\}$. Set
\[
\left(f\left(x,a\right),\ell\left(x,a\right)\right)=\begin{cases}
\left(f_{i}\left(x,\left(a_{i},i\right)\right)e_{i},\ell_{i}\left(x,\left(a_{i},i\right)\right)\right) & \text{if }x\in\Gamma_{i}\backslash\left\{ O\right\} \text{ and }a=\left(a_{i},i\right)\in A_{i},\\
\left(f_{i}\left(O,\left(a_{i},i\right)\right)e_{i},\ell_{i}\left(O,\left(a_{i},i\right)\right)\right) & \text{if }x=O\text{ and }a=\left(a_{i},i\right)\in A_{i},
\end{cases}
\]
where $f_{i}\left(x,\left(a_{i},i\right)\right)=a_{i}$ and $\ell_{1}\equiv1,\ell_{2}\left(x,\left(a_{2},2\right)\right)=1-a_{2}$.
For $x\in \Gamma_{2}\backslash\left\{ O\right\} $, then 
$\mathsf{v}\left(x\right)=v_{2}\left(x\right)=0$ with optimal strategy consists in choosing $\alpha\left(t\right)\equiv\left(1,2\right)$. For $x\in \Gamma_{1}$, we can check that
 $\mathsf{v}\left(x\right)=\min\left\{ \dfrac{1}{\lambda},\dfrac{1-e^{-\lambda\left|x\right|}}{\lambda}+c_{2}e^{-\lambda\left|x\right|}\right\} $. More precisely, for all $x\in \Gamma_{1}$, we have
\[
\mathsf{v}\left(x\right)=\begin{cases}
\dfrac{1}{\lambda} & \text{if }c_{2}\ge\dfrac{1}{\lambda},\text{ with the optimal control  \ensuremath{\alpha\left(t\right)\equiv\left(-1,1\right)}},\\
\dfrac{1-e^{-\lambda\left|x\right|}}{\lambda}+c_{2}e^{-\lambda\left|x\right|} & \text{if }c_{2}<\dfrac{1}{\lambda},\text{ with the optimal control  \ensuremath{\alpha\left(t\right)=\begin{cases}
\left(1,1\right) & \text{if }t\le\left|x\right|,\\
\left(1,2\right) & \text{if }t\ge\left|x\right|.
\end{cases}}}
\end{cases}
\]
Summarizing, we have the two following cases
\begin{enumerate}
  \item If $c_{2}\ge\dfrac{1}{\lambda}$, then 
\[
\mathsf{v}\left(x\right)=\begin{cases}
0 & \quad\mbox{if }x\in \Gamma_{2}\backslash\left\{ O\right\} ,\\
\dfrac{1}{\lambda} & \quad\mbox{if }x\in \Gamma_{1}.
\end{cases}
\]
The graph of the value function with entry costs $c_{2}\ge\dfrac{1}{\lambda}=1$
is plotted in Figure~\ref{fig:Discontinuous function 1}.
\item  If $c_{2}<\dfrac{1}{\lambda}$, then 
\[
\mathsf{v}\left(x\right)=\begin{cases}
0 & \quad\mbox{if }x\in \Gamma_{2}\backslash\left\{ O\right\} ,\\
\dfrac{1-e^{-\lambda\left|x\right|}}{\lambda}+c_{2}e^{-\lambda\left|x\right|} & \quad\mbox{if }x\in \Gamma_{1}.
\end{cases}
\]
The graph of the value function with entry costs $c_{2}=\dfrac{1}{2}<1=\dfrac{1}{\lambda}$
is plotted in Figure~\ref{fig:Discontinuous function 2}.
\end{enumerate}
\begin{figure}
\subfloat[The value function with entry
cost $c_{2}\ge\dfrac{1}{\lambda}=1$.\label{fig:Discontinuous function 1}]
  {\includegraphics[width=.5\linewidth]{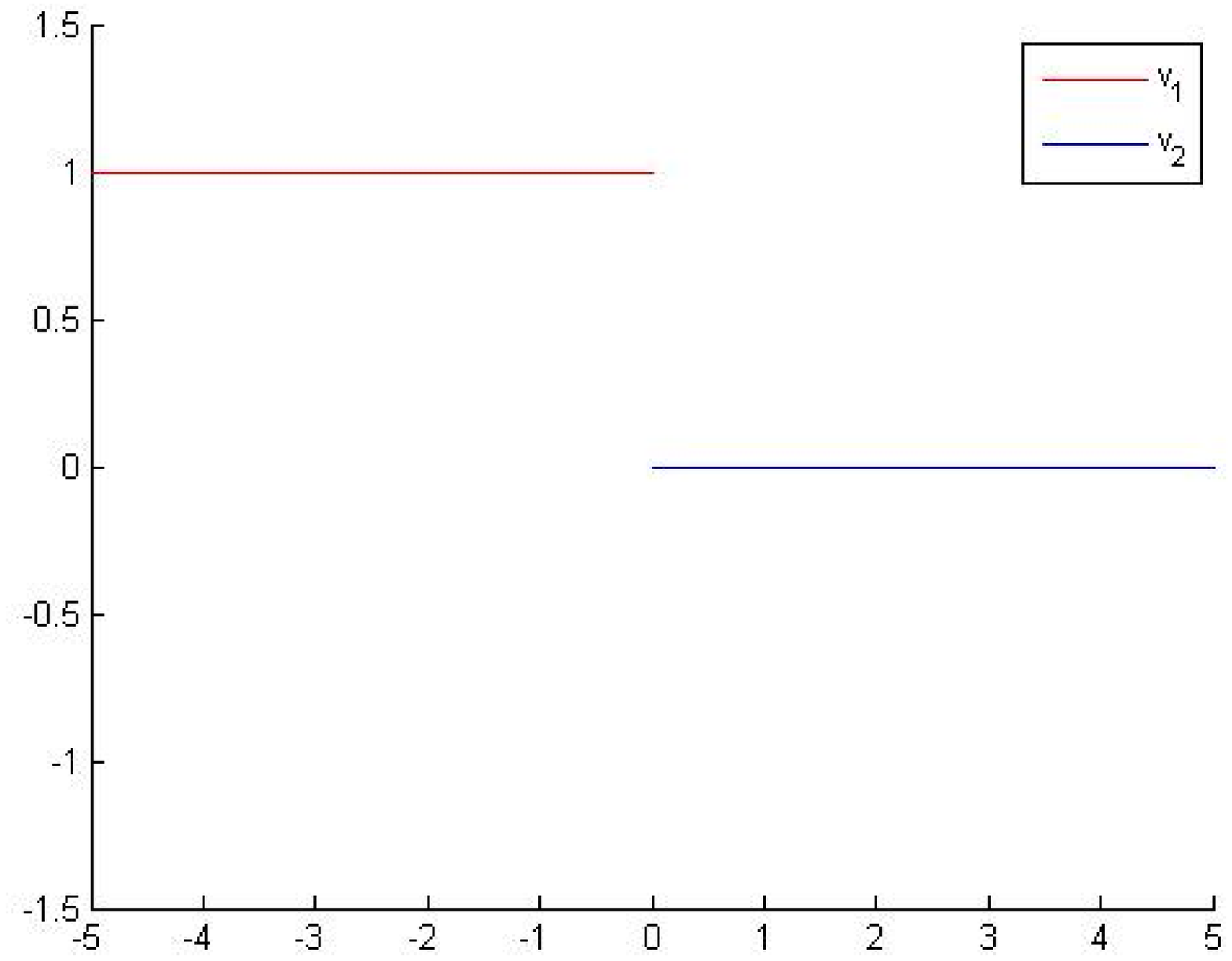}}\hfill
\subfloat[The value function with entry
cost $c_{2}=\dfrac{1}{2}<1=\dfrac{1}{\lambda}$.  \label{fig:Discontinuous function 2}]
  {\includegraphics[width=.5\linewidth]{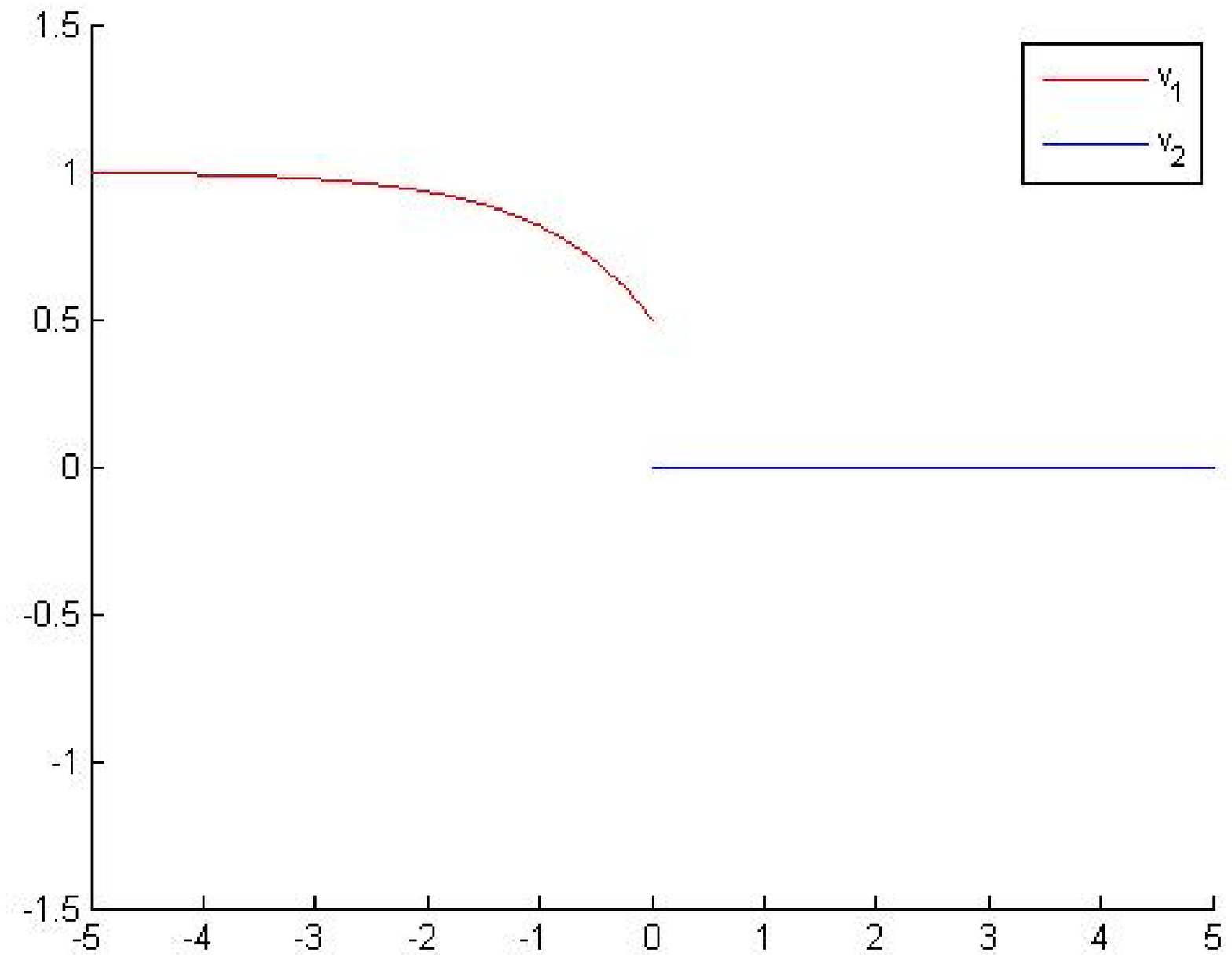}}\hfill

\caption{An example of value function with entry cost}
\end{figure}

\end{example}
\begin{lem} \label{Control t at x near O}Under assumptions $\left[H1\right]$ and $\left[H4\right]$,
there exist two  positive numbers $r_{0}$ and  $C$ such that for
all $x_{1},x_{2}\in B\left(O,r_{0}\right)\cap\mathcal{G}$, there
exists $\left(y_{x_{1},\alpha_{x_{1},x_{2}}},\alpha_{x_{1},x_{2}}\right)\in\mathcal{T}_{x_{1}}$
and $\tau_{x_{1},x_{2}}\le Cd\left(x_{1},x_{2}\right)$
such that $y_{x_{1}}\left(\tau_{x_{1},x_{2}}\right)=x_{2}$.
\end{lem}
\begin{proof}[Proof of Lemma~\ref{Control t at x near O}]
This proof is classical. It is sufficient to consider the case when
$x_{1}$ and $x_{2}$ belong to same edge $\Gamma_{i}$, since in the other
cases, we will use $O$ as a connecting point between $x_{1}$ and
$x_{2}$. According to Assumption $\left[H4\right]$, there exists
$a\in A_{i}$ such that $f_{i}\left(O,a\right)=\delta$. Additionally,
by the Lipschitz continuity of $f_{i}$,
\[
\left|f_{i}\left(O,a\right)-f_{i}\left(x,a\right)\right|\le L\left|x\right|,
\]
hence, if we choose $r_{0}:=\dfrac{\delta}{2L}>0$, then $f_{i}\left(x,a\right)\ge\dfrac{\delta}{2}$
for all $x\in B\left(O,r_{0}\right)\cap \Gamma_{i}$. Let $x_{1},x_{2}$
be in $B\left(O,r_{0}\right)\cap \Gamma_{i}$ with $\left|x_{1}\right|<\left|x_{2}\right|$:
there exist a control law $\alpha$ and $\tau_{x_{1},x_{2}}>0$ such
that $\alpha\left(t\right)=a$ if $0\le t\le\tau_{x_{1},x_{2}}$ and
$y_{x_{1},\alpha}\left(\tau_{x_{1},x_{2}}\right)=x_{2}$. Moreover,
since the velocity $f_{i}\left(y_{x_{1},\alpha}\left(t\right),\alpha\left(t\right)\right)$
is always greater than $\dfrac{\delta}{2}$ when $t\le\tau_{x_{1},x_{2}}$,
then $\tau_{x_{1},x_{2}}\le\dfrac{2}{\delta}d\left(x_{1},x_{2}\right).$
If $\left|x_{1}\right|>\left|x_{2}\right|$,  the proof is achieved by replacing $a\in A_{i}$
by $\overline{a}\in A_{i}$ such that $f_{i}\left(O,\overline{a}\right)=-\delta$
and applying the same argument as above.
\end{proof}

\subsection{\label{subsec: value function at O}Some properties of value function at the vertex}
\begin{lem}
\label{lem: extension value function}Under assumption $\left[H\right]$, 
 $\mathsf{v}|_{\Gamma_{i}\backslash\left\{ O\right\} }$
and $\widehat{\mathsf{v}}|_{\Gamma_{i}\backslash\left\{ O\right\} }$ are continuous
for any $i=\overline{1,N}$. Moreover, there exists $\varepsilon>0$
such that $\mathsf{v}|_{\Gamma_{i}\backslash\left\{ O\right\}}$ and $\widehat{\mathsf{v}}|_{\Gamma_{i}\backslash\left\{ O\right\}}$
are Lipschitz continuous in $\left(\Gamma_{i}\backslash\left\{ O\right\} \right)\cap B\left(O,\varepsilon\right)$.
Hence, it is possible to extend $\mathsf{v}|_{\Gamma_{i}\backslash\left\{ O\right\} }$
and $\widehat{\mathsf{v}}|_{\Gamma_{i}\backslash\left\{ O\right\} }$ at $O$ into Lipschitz continuous functions in $\Gamma_i \cap B\left(O,\varepsilon\right)$. Hereafter,  
$v_{i}$ and $\widehat{v}_{i}$ denote these extensions. 
\end{lem}
\begin{proof}[Proof of Lemma~\ref{lem: extension value function}]
The proof of continuity inside the edge  is classical by using $\left[H4\right]$, see \cite{ACCT2011} for more details. The proof of Lipschitz continuity is a consequence of Lemma~\ref{Control t at x near O}. Indeed, for $x,y$ belong to $\Gamma_{i}\cap B\left(0,\varepsilon\right)$, by Lemma~\ref{Control t at x near O} and the definition of value function, we have
\[
\mathsf{v}\left(x\right)-\mathsf{v}\left(z\right) = v_{i}\left(x\right)-v_{i}\left(z\right)\le\int_{0}^{\tau_{x,z}}\ell_{i}\left(y_{x,\alpha_{x,z}}\left(t\right),\alpha_{x,z}\left(t\right)\right)e^{-\lambda t}dt+v_{i}\left(z\right)\left(e^{-\lambda\tau_{x,z}}-1\right).
\]
Since $\ell_{i}$ is bounded by $M$ (by $\left[H2\right]$), $v_{i}$
is bounded in $\Gamma_{i}\cap B\left(O,\varepsilon\right)$ and $e^{-\lambda\tau_{x,z}}-1$
is bounded by $\tau_{x,y}$, there exists a constant $\overline{C}$
such that
\[
v_{i}\left(x\right)-v_{i}\left(z\right)\le\overline{C}\tau_{x,z}\le\overline{C}C\left|x-z\right|.
\]
The last inequality follows from the Lemma \ref{Control t at x near O}. The inequality $v_{i}\left(z\right)-v_{i}\left(x\right)\le\overline{C}C\left|x-z\right|$
is obtained in a similar way. The proof is done.
\end{proof}
Let us define the tangential Hamiltonian $H_{O}^{T}$ at vertex $O$
by
\begin{equation}
H_{O}^{T}={\displaystyle \max_{i=\overline{1,N} }\max_{a_{i}\in A_{i}^{O}}\left\{ -\ell_{j}\left(O,a_{j}\right)\right\} }={\displaystyle -\min_{i=\overline{1,N} }\min_{a_{i}\in A_{i}^{O}}\left\{ \ell_{j}\left(O,a_{j}\right)\right\} },\label{eq:2}
\end{equation}
where $A_{i}^{O}=\left\{ a_{i}\in A_{i}:f_{i}\left(O,a_{i}\right)=0\right\} .$
The relationship between the values $\mathsf{v}(O)$, $v_{i}\left(O\right)$ and $H_{O}^{T}$ will be given 
in the next theorem.
Hereafter, the proofs of the results will be supplied only for the value function with entry costs
 $\mathsf{v}$, the proofs concerning the value
function with exit costs $\widehat{\mathsf{v}}$ are totally similar.
\begin{thm}
\label{main theorem value function} Under assumption $\left[H\right]$, 
the value functions $\mathsf{v}$ and $\widehat{\mathsf{v}}$ satisfy
\[
\mathsf{v}\left(O\right)=\min\left\{ \min_{i=\overline{1,N}}\left\{ v_{i}\left(O\right)+c_{i}\right\} ,-\dfrac{H_{O}^{T}}{\lambda}\right\} ,
\]
and
\[
\widehat{\mathsf{v}}\left(O\right)=\min\left\{ \min_{i=\overline{1,N}}\left\{ \widehat{v}_{i}\left(O\right)\right\} ,-\dfrac{H_{O}^{T}}{\lambda}\right\} .
\]
\end{thm}
\begin{rem}
Theorem \ref{main theorem value function} gives us the characterization
of the value function at vertex $O$. 
\end{rem}
The proof of Theorem~\ref{main theorem value function}, makes use of  Lemma~\ref{value function at O (1)}
and Lemma~\ref{value function at O (2)} below. 
\begin{lem}[Value functions $\mathsf{v}$ and $\widehat{\mathsf{v}}$ at $O$]
\label{value function at O (1)}Under assumption $\left[H\right]$, then
\[
\max_{i=\overline{1,N}}\left\{ v_{i}\left(O\right)\right\} \le \mathsf{v}\left(O\right)\le\min_{i=\overline{1,N}}\left\{ v_{i}\left(O\right)+c_{i}\right\} ,
\]
and
\[
\max_{i=\overline{1,N}}\left\{ \widehat{v}_{i}\left(O\right)-d_{i}\right\} \le\widehat{\mathsf{v}}\left(O\right)\le\min_{i=\overline{1,N}}\left\{ \widehat{v}_{i}\left(O\right)\right\} .
\]
\end{lem}
\begin{proof}[Proof of Lemma~\ref{value function at O (1)}]
We divide the proof into two parts.

\item\emph{Prove that $\max_{i=\overline{1,N}}\left\{ v_{i}\left(O\right)\right\} \le \mathsf{v}\left(O\right)$}. 
First, we fix $i\in\left\{ 1,\ldots,N\right\} $ and any control law
$\overline{\alpha}$ such that $\left(y_{O,\bar{\alpha}},\bar{\alpha}\right)\in\mathcal{T}_{O}$.
Let $x\in \Gamma_{i}\backslash\left\{ O\right\} $ such that $\left|x\right|$
is small. From Lemma~\ref{Control t at x near O}, there exists
a control law $\alpha_{x,O}$ connecting $x$ and $O$ and we consider
\[
\alpha\left(s\right)=\begin{cases}
\alpha_{x,O}\left(s\right) & \quad\mbox{if }s\le\tau_{x,O},\\
\bar{\alpha}\left(s-\tau_{x,O}\right) & \quad\mbox{if }s>\tau_{x,O}.
\end{cases}
\]
It means that the trajectory goes from $x$
to $O$ with the control law $\alpha_{x,O}$ and then proceeds with the control
law $\bar{\alpha}$. Therefore
\begin{align*}
\mathsf{v}\left(x\right) & =v_{i}\left(x\right)\le J\left(x,\alpha\right)=\int_{0}^{\tau_{x,O}}\ell_{i}\left(y_{x,\alpha}\left(s\right)\right)e^{-\lambda s}ds+e^{-\lambda\tau_{x,O}}J\left(O,\bar{\alpha}\right).
\end{align*}
Since $\overline{\alpha}$ is chosen arbitrarily and $\ell_{i}$ is
bounded by $M$, we get
\[
v_{i}\left(x\right)\le M\tau_{x,O}+e^{-\lambda\tau_{x,O}}\mathsf{v}\left(O\right). 
\]
Let $x$ tend to $O$ then $\tau_{x,O}$ tend to $0$ from Lemma
\ref{Control t at x near O}. Therefore, $v_{i}\left(O\right)\le \mathsf{v}\left(O\right)$.
Since the above inequality holds for $i=\overline{1,N} $,
we obtain that
\[
\max_{i=\overline{1,N}}\left\{ v_{i}\left(O\right)\right\} \le \mathsf{v}\left(O\right).
\]
\item\emph{Prove that $\mathsf{v}\left(O\right)\le\min_{i=\overline{1,N}}\left\{ v_{i}\left(O\right)+c_{i}\right\} $}.
For $i=\overline{1,N} $; we claim that $\mathsf{v}\left(O\right)\le v_{i}\left(O\right)+c_{i}$.
Consider $x\in \Gamma_{i}\backslash\left\{ O\right\} $ with $\left|x\right|$
 small enough and any control law $\bar{\alpha}_{x}$ such that $\left(y_{x,\bar{\alpha}_{x}},\bar{\alpha}_{x}\right)\in\mathcal{T}_{x}$.
From Lemma~\ref{Control t at x near O}, there exists a control
law $\alpha_{O,x}$ connecting $O$ and $x$ and we consider 
\[
\alpha\left(s\right)=\begin{cases}
\alpha_{O,x}\left(s\right) & \quad\mbox{if }s\le\tau_{O,x},\\
\bar{\alpha}_{x}\left(s-\tau_{O,x}\right) & \quad\mbox{if }s>\tau_{O,x}.
\end{cases}
\]
It means that the trajectory goes from $O$ to $x$ using the  control law $\alpha_{O,x}$
then proceeds with the control law $\bar{\alpha}_{x}$. Therefore
\[
\mathsf{v}\left(O\right)\le J\left(O,\alpha\right)=c_{i}+\int_{0}^{\tau_{O,x}}\ell_{i}\left(y_{O,\alpha}\left(s\right)\right)e^{-\lambda s}ds+e^{-\lambda\tau_{O,x}}J\left(x,\bar{\alpha}_{x}\right).
\]
Since $\overline{\alpha}_{x}$ is chosen arbitrarily and $\ell_{i}$
is bounded by $M$, we get
\[
\mathsf{v}\left(O\right)  \le  c_{i}+M\tau_{O,x}+e^{-\lambda\tau_{O,x}}v_{i}\left(x\right)
\]
Let $x$ tend to $O$ then $\tau_{O,x}$ tends to $0$ from Lemma
\ref{Control t at x near O}, then $\mathsf{v}\left(O\right)\le c_{i}+v_{i}\left(O\right).$
Since the above inequality holds for $i=\overline{1,N} $,
we obtain that
\[
\mathsf{v}\left(O\right)\le\min_{i=\overline{1,N}}\left\{ v_{i}\left(O\right)+c_{i}\right\} .
\]
\end{proof}
\begin{lem}
\label{value function at O (2)}The value functions $\mathsf{v}$
and $\widehat{\mathsf{v}}$ satisfy
\begin{equation}
  \label{eq:1}
 \mathsf{v}\left(O\right),\widehat{\mathsf{v}}\left(O\right) \le-\dfrac{H_{O}^{T}}{\lambda}
\end{equation}
where $H_{O}^{T}$ is defined in~\eqref{eq:2}.

\end{lem}
\begin{proof}[Proof of Lemma~\ref{value function at O (2)}]
From~\eqref{eq:2}, there exists $j\in\left\{ 1,\ldots,N\right\} $ and $a_{j}\in A_{j}^{O}$
such that
\[
H_{O}^{T}=-\min_{i=\overline{1,N}}\min_{a_{i}\in A_{i}^{O}}\left\{ \ell_{i}\left(O,a_{i}\right)\right\} =-\ell_{j}\left(O,a_{j}\right)
\]
Let the control law $\alpha$ be defined by $\alpha\left(s\right)\equiv a_{j}$
for all $s$, then
\[
\mathsf{v}\left(O\right)\le J\left(O,\alpha\right)=\int_{0}^{+\infty}\ell_{j}\left(O,a_{j}\right)e^{-\lambda s}ds=\dfrac{\ell_{j}\left(O,a_{j}\right)}{\lambda}=-\dfrac{H_{O}^{T}}{\lambda}.
\]
\end{proof}
We are ready to prove Theorem~\ref{main theorem value function}.
\begin{proof}[Proof of Theorem~\ref{main theorem value function}]
According to Lemma~\ref{value function at O (1)} and Lemma~\ref{value function at O (2)},
\[
\mathsf{v}\left(O\right)\le\min\left\{ \min_{i=\overline{1,N}}\left\{ v_{i}\left(O\right)+c_{i}\right\} ,-\dfrac{H_{O}^{T}}{\lambda}\right\} .
\]
Assuming that
\begin{equation}
\mathsf{v}\left(O\right)<\min_{i=\overline{1,N}}\left\{ v_{i}\left(O\right)+c_{i}\right\} ,\label{recover value function}
\end{equation}
it is sufficient to prove that 
$\mathsf{v}\left(O\right)=-\dfrac{H_{O}^{T}}{\lambda}$.
By~\eqref{recover value function}, there exists a sequence $\left\{ \varepsilon_{n}\right\} _{n\in\mathbb{N}}$
such that $\varepsilon_{n}\rightarrow0$ and 
\[
\mathsf{v}\left(O\right)+\varepsilon_{n}<\min_{i=\overline{1,N}}\left\{ v_{i}\left(O\right)+c_{i}\right\}\quad\mbox{for all } n\in\mathbb{N}.
\]
On the other hand, there exists an $\varepsilon_{n}$-optimal
control $\alpha_{n}$, $\mathsf{v}\left(O\right)+\varepsilon_{n}>J\left(O,\alpha_{n}\right)$.
Let us define the first time that the  trajectory $y_{O,\alpha_{n}}$ leaves
$O$
\[
t_{n}:=\inf_{i=\overline{1,N} }T_{i}^{n},
\]
where $T_{i}^{n}$ is the set of times $t$ for which $y_{O,\alpha_{n}}(t)$
belongs to $\Gamma_{i}\backslash\left\{ O\right\} $. Notice that $t_{n}$
is possibly $+\infty$, in which case $y_{O,\alpha_{n}}\left(s\right)=O$ for all $s\in \left[0,+\infty\right)$. 
Extracting a subsequence if necessary, we may assume that $t_{n}$ tends to $\overline{t}\in\left[0,+\infty\right]$ when $\varepsilon_{n}$ tends to $0$.

If there exists a subsequence of $\left\{ t_{n}\right\} _{n\in\mathbb{N}}$
(which is still noted  $\left\{ t_{n}\right\} _{n\in\mathbb{N}}$) such
that $t_{n}=+\infty$ for all $n\in\mathbb{N}$,
 then for a.e. $s\in\left[0,+\infty\right)$
\[
\begin{cases}
f\left(y_{O,\alpha_{n}}\left(s\right),\alpha_{n}\left(s\right)\right) & =f\left(O,\alpha_{n}\left(s\right)\right)=0,\\
\ell\left(y_{O,\alpha_{n}}\left(s\right),\alpha_{n}\left(s\right)\right) & =\ell\left(O,\alpha_{n}\left(s\right)\right).
\end{cases}
\]
In this case, $\alpha_{n}\left(s\right)\in\cup_{i=1}^{N}A_{i}^{O}$
for a.e. $s\in\left[0,+\infty\right)$. Therefore, for a.e. $s\in\left[0,+\infty\right)$
\[
\ell\left(y_{O,\alpha_{n}}\left(s\right),\alpha_{n}\left(s\right)\right)=\ell\left(O,\alpha_{n}\left(s\right)\right)\ge-H_{O}^{T},
\]
and
\[
\mathsf{v}\left(O\right)+\varepsilon_{n}>J\left(O,\alpha_{n}\right)=\int_{0}^{+\infty}\ell\left(O,\alpha_{n}\left(s\right)\right)e^{-\lambda s}ds\ge\int_{0}^{+\infty}\left(-H_{O}^{T}\right)e^{-\lambda s}ds=-\dfrac{H_{O}^{T}}{\lambda}.
\]
By letting $n$ tend to $\infty$, we get $\mathsf{v}\left(O\right)\ge-\dfrac{H_{O}^{T}}{\lambda}$.
On the other hand, since $\mathsf{v}\left(O\right)\le-\dfrac{H_{O}^{T}}{\lambda}$
by Lemma~\ref{value function at O (2)}, this implies that $\mathsf{v}\left(O\right)=-\dfrac{H_{O}^{T}}{\lambda}$.

Let us now assume that $0\le t_{n}<+\infty$ for all $n$ large enough. Then, for a fixed $n$ and for any positive $\delta\le\delta_{n}$ where $\delta_{n}$ small enough, $y_{O,\alpha_{n}}\left(s\right)$ still belongs to some $\Gamma_{i\left(n\right)}\backslash\left\{ O\right\} $
for all $s\in\left(t_{n},t_{n}+\delta\right]$. We have
\begin{eqnarray*}
\mathsf{v}\left(O\right)+\varepsilon_{n} & > & J\left(O,\alpha_{n}\right)\\
 & = & \int_{0}^{t_{n}}\ell\left(y_{O,\alpha_{n}}\left(s\right),\alpha_{n}\left(s\right)\right)e^{-\lambda s}ds+c_{i\left(n\right)}e^{-\lambda t_{n}}+\int_{t_{n}}^{t_{n}+\delta}\ell_{i\left(n\right)}\left(y_{O,\alpha_{n}}\left(s\right),\alpha_{n}\left(s\right)\right)e^{-\lambda s}ds\\
 &  & +e^{-\lambda\left(t_{n}+\delta\right)}J\left(y_{O,\alpha_{n}}\left(t_{n}+\delta\right),\alpha_{n}\left(\cdot+t_{n}+\delta\right)\right)\\
 & \ge & \int_{0}^{t_{n}}\ell\left(y_{O,\alpha_{n}}\left(s\right),\alpha_{n}\left(s\right)\right)e^{-\lambda s}ds+c_{i\left(n\right)}e^{-\lambda t_{n}}+\int_{t_{n}}^{t_{n}+\delta}\ell_{i\left(n\right)}\left(y_{O,\alpha_{n}}\left(s\right),\alpha_{n}\left(s\right)\right)e^{-\lambda s}ds\\
 &  & +e^{-\lambda\left(t_{n}+\delta\right)}v\left(y_{O,\alpha_{n}}\left(t_{n}+\delta\right)\right)\\
 & = & \int_{0}^{t_{n}}\ell\left(y_{O,\alpha_{n}}\left(s\right),\alpha_{n}\left(s\right)\right)e^{-\lambda s}ds+c_{i\left(n\right)}e^{-\lambda t_{n}}+\int_{t_{n}}^{t_{n}+\delta}\ell_{i\left(n\right)}\left(y_{O,\alpha_{n}}\left(s\right),\alpha_{n}\left(s\right)\right)e^{-\lambda s}ds\\
 &  & +e^{-\lambda\left(t_{n}+\delta\right)}v_{i\left(n\right)}\left(y_{O,\alpha_{n}}\left(t_{n}+\delta\right)\right).
\end{eqnarray*}
By letting $\delta$ tend to $0$, 
\[
\mathsf{v}\left(O\right)+\varepsilon_{n}\ge\int_{0}^{t_{n}}\ell\left(y_{O,\alpha_{n}}\left(s\right),\alpha_{n}\left(s\right)\right)e^{-\lambda s}ds+c_{i\left(n\right)}e^{-\lambda t_{n}}+e^{-\lambda t_{n}}v_{i\left(n\right)}\left(O\right).
\]
Note that $y_{O,\alpha_{n}}\left(s\right)=O$ for all $s\in\left[0,t_{n}\right]$,
i.e., $f\left(O,\alpha_{n}\left(s\right)\right)=0$ a.e. $s\in\left[0,t_{n}\right)$.
Hence
\begin{eqnarray*}
\mathsf{v}\left(O\right)+\varepsilon_{n} & \ge & \int_{0}^{t_{n}}\ell\left(O,\alpha_{n}\left(s\right)\right)e^{-\lambda s}ds+c_{i\left(n\right)}e^{-\lambda t_{n}}+e^{-\lambda t_{n}}v_{i\left(n\right)}\left(O\right)\\
 & \ge & \int_{0}^{t_{n}}\left(-H_{O}^{T}\right)e^{-\lambda s}ds+c_{i\left(n\right)}e^{-\lambda t_{n}}+e^{-\lambda t_{n}}v_{i\left(n\right)}\left(O\right)\\
 & = & \dfrac{1-e^{-\lambda t_{n}}}{\lambda}\left(-H_{O}^{T}\right)+c_{i\left(n\right)}e^{-\lambda t_{n}}+e^{-\lambda t_{n}}v_{i\left(n\right)}\left(O\right).
\end{eqnarray*}
Choose a subsequence $\left\{ \varepsilon_{n_{k}}\right\} _{k\in\mathbb{N}}$
of $\left\{ \varepsilon_{n}\right\} _{n\in\mathbb{N}}$ such that 
for some $i_{0}\in\left\{ 1,\ldots,N\right\} $, $c_{i\left(n_{k}\right)}=c_{i_{0}}$ for all $k$. By letting
$k$ tend to $\infty$, recall that $\lim_{k\rightarrow\infty}t_{n_{k}}=\overline{t}$, we have three possible cases
\begin{enumerate}
  \item If $\overline{t}=+\infty$, then $\mathsf{v}\left(O\right)\ge-\dfrac{H_{O}^{T}}{\lambda}.$
By Lemma~\ref{value function at O (2)}, we obtain $\mathsf{v}\left(O\right)=-\dfrac{H_{O}^{T}}{\lambda}$.
\item  If $\overline{t}=0$, then $\mathsf{v}\left(O\right)\ge c_{i_{0}}+v_{i_{0}}\left(O\right)$.
By~\eqref{recover value function}, we obtain a  contradiction.
\item If $\overline{t}\in\left(0,+\infty\right)$, then $\mathsf{v}\left(O\right)\ge\dfrac{1-e^{-\lambda\overline{t}}}{\lambda}\left(-H_{O}^{T}\right)+\left[c_{i_{0}}+v_{i_{0}}\left(O\right)\right]e^{-\lambda\overline{t}}.$
By~\eqref{recover value function}, $c_{i_{0}}+v_{i_{0}}\left(O\right)>\mathsf{v}\left(O\right)$,
so
\[
\mathsf{v}\left(O\right)>\dfrac{1-e^{-\lambda\overline{t}}}{\lambda}\left(-H_{O}^{T}\right)+\mathsf{v}\left(O\right)e^{-\lambda\overline{t}}.
\]
This yields $\mathsf{v}\left(O\right)>-\dfrac{H_{O}^{T}}{\lambda}$, and finally obtain a contradiction by
Lemma~\ref{value function at O (2)}. 
\end{enumerate}
\end{proof}
\section{The Hamilton-Jacobi systems. Viscosity solutions}
\label{sec:hamilt-jacobi-syst}
\subsection{Test-functions}
\begin{defn}
A function $\varphi:\Gamma_{1}\times\ldots\times \Gamma_{N}\rightarrow\mathbb{R}^{N}$ is an admissible test-function if there exists 
$(\varphi_{i})_{i=\overline{1,N}}$,   $\varphi_{i} \in C^{1}\left(\Gamma_{i}\right)$, such that
 $\varphi\left(x_{1},\ldots,x_{N}\right)=\left(\varphi_{1}\left(x_{1}\right),\ldots,\varphi_{N}\left(x_{N}\right)\right)$.
The set of admissible test-function is denoted by $\mathcal{R}\left(\mathcal{G}\right)$.
\end{defn}

\subsection{Definition of viscosity solution}
\begin{defn}[Hamiltonian]
We define the Hamiltonian $H_{i}:\Gamma_{i}\times\mathbb{R}\rightarrow\mathbb{R}$
by
\[
H_{i}\left(x,p\right)=\max_{a\in A_{i}}\left\{ -pf_{i}\left(x,a\right)-\ell_{i}\left(x,a\right)\right\} 
\]
and the Hamiltonian $H_{i}^{+} (O,\cdot) : \mathbb{R}\rightarrow\mathbb{R}$
by
\[
H_{i}^{+}\left(O,p\right)=\max_{a\in A_{i}^{+}}\left\{ -pf_{i}\left(O,a\right)-\ell_{i}\left(O,a\right)\right\} ,
\]
where $A_{i}^{+}=\left\{ a_{i}\in A_{i}:f_{i}\left(O,a_{i}\right)\ge0\right\} $.
Recall that  {\sl the tangential Hamiltonian at $O$}, $H_{O}^{T}$, has been defined in~\eqref{eq:2}.
\end{defn}
We now introduce the Hamilton-Jacobi system for the case with entry costs
\begin{equation}
\begin{array}{cc}
\lambda u_{i}\left(x\right)+H_{i}\left(x,\dfrac{d u_{i}}{d x_{i}}\left(x\right)\right)=0 & \mbox{if \ensuremath{x\in \Gamma_{i}\backslash \left\{O\right\}},}\\
{\displaystyle \lambda u_{i}\left(O\right)+\max\left\{ -\lambda\min_{j\ne i}\left\{ u_{j}\left(O\right)+c_{j}\right\} ,H_{i}^{+}\left(O,\dfrac{d u_{i}}{d x_{i}}\left(O\right)\right),H_{O}^{T}\right\} =0} & \mbox{if \ensuremath{x=O},}
\end{array}\label{eq: Hamilton Jacobi entry}
\end{equation}
for all $i=\overline{1,N}$ and the Hamilton-Jacobi system with exit
costs
\begin{equation}
\begin{array}{cc}
\lambda\widehat{u}_{i}\left(x\right)+H_{i}\left(x,\dfrac{d\widehat{u}_{i}}{d x_{i}}\left(x\right)\right)=0 & \mbox{if \ensuremath{x\in \Gamma_{i}\backslash \left\{O\right\}},}\\
{\displaystyle \lambda\widehat{u}_{i}\left(O\right)+\max\left\{ -\lambda\min_{j\ne i}\left\{ \widehat{u}_{j}\left(O\right)+d_{i}\right\} ,H_{i}^{+}\left(O,\dfrac{d\widehat{u}_{i}}{dx_{i}}\left(O\right)\right),H_{O}^{T}-\lambda d_{i}\right\} =0} & \mbox{if \ensuremath{x=O}, }
\end{array}\label{eq: Hamilton Jacobi exit}
\end{equation}
for all $i=\overline{1,N}$ and their viscosity solutions.
\begin{defn}[Viscosity solution with entry costs]
$ $

$\bullet$ A function $u:=\left(u_{1},\ldots,u_{N}\right)$
where $u_{i}\in USC\left(\Gamma_{i};\mathbb{R}\right)$ for all $i=\overline{1,N}$, is
called a \emph{viscosity sub-solution} of~\eqref{eq: Hamilton Jacobi entry}
if for any $\left(\varphi_{1},\ldots,\varphi_{N}\right)\in\mathcal{R}\left(\mathcal{G}\right)$,
any $i=\overline{1,N}$ and any $x_{i}\in \Gamma_{i}$ such that $u_{i}-\varphi_{i}$
has a \emph{local maximum point} on $\Gamma_{i}$ at $x_{i}$, then
\[
\begin{array}{cc}
\lambda u_{i}\left(x_{i}\right)+H_{i}\left(x,\dfrac{d\varphi_{i}}{d x_{i}}\left(x_{i}\right)\right)\le0 & \mbox{if \ensuremath{x_{i}\in \Gamma_{i}\backslash \left\{O\right\}},}\\
{\displaystyle \lambda u_{i}\left(O\right)+\max\left\{ -\lambda\min_{j\ne i}\left\{ u_{j}\left(O\right)+c_{j}\right\} ,H_{i}^{+}\left(O,\dfrac{d\varphi_{i}}{dx_{i}}\left(O\right)\right),H_{O}^{T}\right\} }\le0 & \mbox{if \ensuremath{x_{i}=O}.}
\end{array}
\]

$\bullet$ A function $u:=\left(u_{1},\ldots,u_{N}\right)$
where $u_{i}\in LSC\left(\Gamma_{i};\mathbb{R}\right)$ for all  $i=\overline{1,N}$, is
called a \emph{viscosity super-solution} of~\eqref{eq: Hamilton Jacobi entry}
if for any $\left(\varphi_{1},\ldots,\varphi_{N}\right)\in\mathcal{R}\left(\mathcal{G}\right)$,
any $i=\overline{1,N}$ and any $x_{i}\in \Gamma_{i}$ such that $u_{i}-\varphi_{i}$
has a \emph{local minimum point} on $\Gamma_{i}$ at $x_{i}$, then
\[
\begin{array}{cc}
\lambda u_{i}\left(x_{i}\right)+H_{i}\left(x_{i},\dfrac{d\varphi_{i}}{d x_{i}}\left(x_{i}\right)\right)\ge0 & \mbox{if \ensuremath{x_{i}\in \Gamma_{i}\backslash \left\{O\right\}},}\\
{\displaystyle \lambda u_{i}\left(O\right)+\max\left\{ -\lambda\min_{j\ne i}\left\{ u_{j}\left(O\right)+c_{j}\right\} ,H_{i}^{+}\left(O,\dfrac{d\varphi_{i}}{dx_{i}}\left(O\right)\right),H_{O}^{T}\right\} \ge0} & \mbox{if \ensuremath{x_{i}=O}.}
\end{array}
\]

$\bullet$ A functions $u:=\left(u_{1},\ldots,u_{N}\right)$
where $u_{i}\in C\left(\Gamma_{i};\mathbb{R}\right)$ for all $i=\overline{1,N}$, is called
a \emph{viscosity solution} of~\eqref{eq: Hamilton Jacobi entry}
if it is both a viscosity sub-solution and a viscosity super-solution
of~\eqref{eq: Hamilton Jacobi entry}.
\end{defn}

\begin{defn}[Viscosity solution with exit costs]
$ $

$\bullet$ A function $\widehat{u}:=\left(\widehat{u}_{1},\ldots,\widehat{u}_{N}\right)$
where $\widehat{u}_{i}\in USC\left(\Gamma_{i};\mathbb{R}\right)$ for all $i=\overline{1,N}$,
is called a \emph{viscosity sub-solution} of~\eqref{eq: Hamilton Jacobi exit}
if for any $\left(\psi_{1},\ldots,\psi_{N}\right)\in\mathcal{R}\left(\mathcal{G}\right)$,
any $i=\overline{1,N}$ and any $y_{i}\in \Gamma_{i}$ such that $\widehat{u}_{i}-\psi_{i}$
has a \emph{local maximum point} on $\Gamma_{i}$ at $y_{i}$, then
\[
\begin{array}{cc}
\lambda\widehat{u}_{i}\left(y_{i}\right)+H_{i}\left(y_{i},\dfrac{d\psi_{i}}{d x_{i}}\left(y_{i}\right)\right)\le0 & \mbox{if \ensuremath{y_{i}\in \Gamma_{i}\backslash \left\{O\right\}},}\\
{\displaystyle \lambda\widehat{u}_{i}\left(O\right)+\max\left\{ -\lambda\min_{j\ne i}\left\{ \widehat{u}_{j}\left(O\right)\right\} -\lambda d_{i},H_{i}^{+}\left(O,\dfrac{d\psi_{i}}{dx_{i}}\left(O\right)\right),H_{O}^{T}-\lambda d_{i}\right\} }\le0 & \mbox{if \ensuremath{y_{i}=O}.}
\end{array}
\]

$\bullet$ A function $\widehat{u}:=\left(\widehat{u}_{1},\ldots,\widehat{u}_{N}\right)$
where $\widehat{u}_{i}\in LSC\left(\Gamma_{i};\mathbb{R}\right)$ for all $i=\overline{1,N}$,
is called a \emph{viscosity super-solution} of~\eqref{eq: Hamilton Jacobi exit}
if for any $\left(\psi_{1},\ldots,\psi_{N}\right)\in\mathcal{R}\left(\mathcal{G}\right)$,
any $i=\overline{1,N}$ and any $y_{i}\in \Gamma_{i}$ such that $u_{i}-\psi_{i}$
has a \emph{local minimum point} on $\Gamma_{i}$ at $y_{i}$, then
\[
\begin{array}{cc}
\lambda\widehat{u}_{i}\left(y_{i}\right)+H_{i}\left(y_{i},\dfrac{d\psi_{i}}{d x_{i}}\left(y_{i}\right)\right)\ge0 & \mbox{if \ensuremath{y_{i}\in \Gamma_{i}\backslash \left\{O\right\}},}\\
{\displaystyle \lambda\widehat{u}_{i}\left(O\right)+\max\left\{ -\lambda\min_{j\ne i}\left\{ \widehat{u}_{j}\left(O\right)\right\} -\lambda d_{i},H_{i}^{+}\left(O,\dfrac{d\psi_{i}}{dx_{i}}\left(O\right)\right),H_{O}^{T}-\lambda d_{i}\right\} }\ge0 & \mbox{if \ensuremath{y_{i}=O}.}
\end{array}
\]

$\bullet$ A functions $\widehat{u}:=\left(\widehat{u}_{1},\ldots,\widehat{u}_{N}\right)$
where $\widehat{u}_{i}\in C\left(\Gamma_{i};\mathbb{R}\right)$ for all $i=\overline{1,N}$,
is called a \emph{viscosity solution} of~\eqref{eq: Hamilton Jacobi exit}
if it is both a viscosity sub-solution and a viscosity super-solution
of~\eqref{eq: Hamilton Jacobi exit}.
\end{defn}
\begin{rem}
This notion of viscosity solution is consitent with the one of~\cite{AOT2015}. It can be seen in Section~\ref{sec:gener-case-switch} when all the switching costs are zero, our definition and the one of~\cite{AOT2015} coincide.
\end{rem}

\section{Connections between the value functions and the Hamilton-Jacobi systems. }
\label{sec:relation-between-two}
Let $\mathsf{v}$ be the value function of the optimal control problem with entry costs and $\widehat{\mathsf{v}}$ be
a value function  of the optimal control problem with exit costs.
Recall that $v_{i},\widehat{v}_{i}:\Gamma_{i}\rightarrow\mathbb{R}$ are defined in Lemma~\ref{lem: extension value function} by
\[
{\displaystyle \begin{cases}
v_{i}\left(x\right)=\mathsf{v}\left(x\right)\quad\mbox{if }x\in \Gamma_{i}\backslash\left\{ O\right\} ,\\
v_{i}\left(O\right)=\lim_{\Gamma_{i}\backslash\left\{ O\right\} \ni x\rightarrow O}\mathsf{v}\left(x\right),
\end{cases}}
\quad \hbox{and}\quad
{\displaystyle \begin{cases}
\widehat{v}_{i}\left(x\right)=\widehat{\mathsf{v}}\left(x\right)\quad\mbox{if }x\in \Gamma_{i}\backslash\left\{ O\right\} ,\\
\widehat{v}_{i}\left(O\right)=\lim_{\Gamma_{i}\backslash\left\{ O\right\} \ni x\rightarrow O}\widehat{\mathsf{v}}\left(x\right).
\end{cases}}
\]

We wish to prove that ${v}:=\left(v_{1},v_{2},\ldots,v_{N}\right)$
and $\widehat{{v}}:=\left(\widehat{v}_{1},\ldots,\widehat{v}_{N}\right)$
are respectively  viscosity solutions of~\eqref{eq: Hamilton Jacobi entry}
and~\eqref{eq: Hamilton Jacobi exit}. In fact, since $\mathcal{G}\backslash\left\{ O\right\} $ is a finite union of open intervals in which the classical theory can be applied,
we obtain that $v_{i}$ and $\widehat{v}_{i}$ are viscosity solutions
of 
\[
\lambda u\left(x\right)+H_{i}\left(x,Du\left(x\right)\right)=0\quad\mbox{in }\Gamma_{i}\backslash\left\{ O\right\} .
\]
Therefore, we can restrict ourselves to prove the following theorem.
\begin{thm}
\label{thm: Existence}For $i=\overline{1,N}$, the function $v_{i}$
satisfies
\[
\lambda v_{i}\left(O\right)+\max\left\{ -\lambda\min_{j\ne i}\left\{ v_{j}\left(O\right)+c_{j}\right\} ,H_{i}^{+}\left(O,\dfrac{d v_{i}}{d x_{i}}\left(O\right)\right),H_{O}^{T}\right\} =0
\]
in the viscosity sense.
The function $\widehat{v}_{i}$ satisfies
\[
\lambda\widehat{v}_{i}\left(O\right)+\max\left\{ -\lambda\min_{j\ne i}\left\{ \widehat{v}_{j}\left(O\right)+d_{i}\right\} ,H_{i}^{+}\left(O,\dfrac{d\widehat{v}_{i}}{d x_{i}}\left(O\right)\right),H_{O}^{T}-\lambda d_{i}\right\} =0
\]
in the viscosity sense.
\end{thm}
The proof of Theorem~\ref{thm: Existence} follows from
Lemmas~\ref{sup_property 1} and~\ref{super_property 2} below. We focus on
$v_{i}$ since the proof for $\widehat{v}_{i}$ is similar.
\begin{lem}
\label{sup_property 1}For $i=\overline{1,N}$, the function $v_{i}$ is a viscosity sub-solution of~\eqref{eq: Hamilton Jacobi entry} at ${O}$.
\end{lem}
\begin{proof}[Proof of Lemma~\ref{sup_property 1}]
From Theorem~\ref{main theorem value function},
\[
\lambda v_{i}\left(O\right)+\max\left\{ -\lambda\min_{j\ne i}\left\{ v_{j}\left(O\right)+c_{j}\right\} ,H_{O}^{T}\right\} \le 0.
\]
It is thus sufficient to prove that 
\[
\lambda v_{i}\left(O\right)+H_{i}^{+}\left(O,\dfrac{d v_{i}}{d x_{i}}\left(O\right)\right)\le 0
\]
in the viscosity sense.
Let $a_{i} \in A_i$ be such that $f_{i}\left(O,a_i\right)>0$. Setting $\alpha\left(t\right)\equiv a_{i}$ then $\left(y_{x,\alpha},\alpha\right)\in\mathcal{T}_{x}$
for all $x\in \Gamma_{i}$. Moreover, for all $x\in \Gamma_{i}\backslash\left\{ O\right\} $,
$y_{x,\alpha}\left(t\right)\in \Gamma_{i}\backslash\left\{ O\right\} $
(the trajectory cannot approach $O$ since the speed pushes it away
from $O$ for $y_{x,\alpha}\in \Gamma_{i}\cap B\left(O,r\right)$). Note
that it is not sufficient to choose $a_{i}\in A_{i}$ such that $f\left(O,a_{i}\right)=0$
since it can lead to $f\left(x,a_{i}\right)<0$ for all $x\in \Gamma_{i}\backslash\left\{ O\right\} $.
Next, for $\tau>0$ fixed and any $x\in \Gamma_{i}$, if we choose 
\begin{equation}
\alpha_{x}\left(t\right)=\begin{cases}
\alpha\left(t\right)=a_{i} & 0\le t\le\tau,\\
\hat{a}\left(t-\tau\right) & t\ge\tau,
\end{cases}\label{eq: remain on network}
\end{equation}
 then $y_{x.\alpha_{x}}\left(t\right)\in \Gamma_{i}\backslash\left\{ O\right\} $
for all $t\in\left[0,\tau\right]$. It yields
\[
v_{i}\left(x\right)  \le  J\left(x,\alpha_{x}\right)=\int_{0}^{\tau}\ell_{i}\left(y_{x,\alpha}\left(s\right),a_{i}\right)e^{-\lambda s}ds+e^{-\lambda\tau}J\left(y_{x,\alpha}\left(\tau\right),\widehat{\alpha}\right).
\]
Since this holds for any $\widehat{\alpha}$ ($\alpha_{x}$ is arbitrary
for $t>\tau$), we deduce that
\begin{equation}
v_{i}\left(x\right)\le\int_{0}^{\tau}\ell_{i}\left(y_{x,\alpha_{x}}\left(s\right),a_{i}\right)e^{-\lambda s}ds+e^{-\lambda\tau}v_{i}\left(y_{x,\alpha_{x}}\left(\tau\right)\right).\label{ineq_sub_property 1}
\end{equation}
Since $f_{i}\left(\cdot,a\right)$ is Lipschitz continuous by $\left[H1\right]$, we also have for all $t\in\left[0,\tau\right]$,
\begin{eqnarray*}
\left|y_{x,\alpha_{x}}\left(t\right)-y_{O,\alpha_{O}}\left(t\right)\right| & = & \left|x+\int_{0}^{t}f_{i}\left(y_{x,\alpha}\left(s\right),a_{i}\right)e_{i}ds-\int_{0}^{t}f_{i}\left(y_{O,\alpha}\left(s\right),a_{i}\right)e_{i}ds\right|\\
 & \le & \left|x\right|+L\int_{0}^{t}\left|y_{x,\alpha}\left(s\right)-y_{O,\alpha}\left(s\right)\right|ds,
\end{eqnarray*}
where $\alpha_{0}$ satisfies~\eqref{eq: remain on network} with
$x=O$. According to Gr{\"o}nwall's inequality, 
\[
\left|y_{x,\alpha_{x}}\left(t\right)-y_{O,\alpha_{O}}\left(t\right)\right|\le\left|x\right|e^{Lt},
\]
for $t\in\left[0,\tau\right]$, yielding that $y_{x,\alpha_{x}}\left(t\right)$ tends to $y_{O,\alpha_{O}}\left(t\right)$
when $x$ tends to $O$. Hence, from~\eqref{ineq_sub_property 1},
by letting $x\rightarrow O$, we obtain
\[
v_{i}\left(O\right)\le\int_{0}^{\tau}\ell_{i}\left(y_{O,\alpha_{O}}\left(s\right),a_{i}\right)e^{-\lambda s}ds+e^{-\lambda\tau}v_{i}\left(y_{O,\alpha_{O}}\left(\tau\right)\right).
\]
Let $\varphi$ be a function in $C^{1}\left(\Gamma_{i}\right)$
such that $0=v_{i}\left(O\right)-\varphi\left(O\right)=\max_{\Gamma_{i}}\left(v_{i}-\varphi\right)$.
This yields
\[
\dfrac{\varphi\left(O\right)-\varphi\left(y_{O,\alpha_{O}}\left(\tau\right)\right)}{\tau}\le\dfrac{1}{\tau}\int_{0}^{\tau}\ell_{i}\left(y_{O,\alpha_{O}}\left(s\right),a_{i}\right)e^{-\lambda s}ds+\dfrac{\left(e^{-\lambda\tau}-1\right)v_{i}\left(y_{O,\alpha_{O}}\left(\tau\right)\right)}{\tau}.
\]
By letting $\tau$ tend to $0$, we obtain that
\[
- f_{i}\left(O,a_{i}\right)\dfrac{d\varphi}{d x_{i}}\left(O\right)\le\ell_{i}\left(O,a_{i}\right)-\lambda v_{i}\left(O\right).
\]
Hence,
\[
\lambda v_{i}\left(O\right)+\sup_{a\in A_{i}:f_{i}\left(O,a\right)>0}\left\{ -f_{i}\left(O,a\right)\dfrac{d v_{i}}{d x_{i}}\left(O\right)-\ell_{i}\left(O,a\right)\right\} \le0
\]
in the viscosity sense. Finally, from Corollary~\ref{corollary: max=00003Dsup}
in Appendix, we have
\[
\sup_{a\in A_{i}:f_{i}\left(O,a\right)>0}\left\{ -f_{i}\left(O,a\right)\dfrac{d\varphi_{i}}{d x_{i}}\left(O\right)-\ell_{i}\left(O,a\right)\right\} =\max_{a\in A_{i}:f_{i}\left(O,a\right)\ge0}\left\{ -f_{i}\left(O,a\right)\dfrac{d\varphi_{i}}{d x_{i}}\left(O\right)-\ell_{i}\left(O,a\right)\right\} .
\]
The proof is complete.
\end{proof}
\begin{lem}
\label{super_property 1} If
\begin{equation}
v_{i}\left(O\right)<{\displaystyle \min\left\{ \min_{j\ne i}\left\{ v_{j}\left(O\right)+c_{j}\right\} ,-\dfrac{H_{O}^{T}}{\lambda}\right\} },\label{ineq: min min}
\end{equation}
then there exist $\bar{\tau}>0,r>0$ and $\varepsilon_{0}>0$ such
that for any $x\in\left(\Gamma_{i}\backslash\left\{ O\right\} \right)\cap B\left(O,r\right)$,
any $\varepsilon<\varepsilon_{0}$ and any   $\varepsilon-$optimal
control law $\alpha_{\varepsilon,x}$ for $x$, 
\[
y_{x,\alpha_{\varepsilon,x}}\left(s\right)\in \Gamma_{i}\backslash\left\{ O\right\} ,\quad\mbox{for all }s \in\left[0,\bar{\tau}\right].
\]
\end{lem}
\begin{rem}
Roughly speaking, this lemma takes care of the case $\lambda v_{i}+H_{i}^{+}\left(x,\dfrac{dv_{i}}{dx_{i}}\left(O\right)\right)\le0$, i.e.,
the situation when the trajectory does not leave $\Gamma_{i}$, see introduction.
\end{rem}
\begin{proof}[Proof of Lemma~\ref{super_property 1}]
Suppose by contradiction that there exist sequences $\left\{ \varepsilon_{n}\right\},\left\{ \tau_{n}\right\}\subset\mathbb{R}^{+}$ and $\left\{ x_{n}\right\}\subset \Gamma_{i}\backslash\left\{ O\right\} $
such that $\varepsilon_{n}\searrow0$, $x_{n}\rightarrow O,\tau_{n}\searrow0$
 and a control law $\alpha_{n}$ such that
$\alpha_{n}$ is $\varepsilon_{n}$-optimal control law and $y_{x_{n},\alpha_{n}}\left(\tau_{n}\right)=O$.
This implies that
\begin{equation}
v_{i}\left(x_{n}\right)+\varepsilon_{n}>J\left(x_{n},\alpha_{n}\right)=\int_{0}^{\tau_{n}}\ell\left(y_{x_{n},\alpha_{n}}\left(s\right),\alpha_{n}\left(s\right)\right)e^{-\lambda s}ds+e^{-\lambda\tau_{n}}J\left(O,\alpha_{n}\left(\cdot+\tau_{n}\right)\right).\label{eq: epsilon-optimal control}
\end{equation}
Since $\ell$ is bounded by $M$ by $\left[H1\right]$, then
$
v_{i}\left(x_{n}\right)+\varepsilon_{n} 
 \ge  -\tau_{n}M+e^{-\lambda\tau_{n}}\mathsf{v}\left(O\right).
$
By letting $n$ tend to $\infty$, we obtain
\begin{equation}
v_{i}\left(O\right)\ge \mathsf{v}\left(O\right).\label{ineq: contradiction}
\end{equation}
From~\eqref{ineq: min min}, it follows that
\[
\min\left\{ \min_{j\ne i}\left\{ v_{j}\left(O\right)+c_{j}\right\} ,-\dfrac{H_{O}^{T}}{\lambda}\right\} >\mathsf{v}\left(O\right).
\]
However, $\mathsf{v}\left(O\right)=\min\left\{ {\displaystyle \min_{j}\left\{ v_{j}\left(O\right)+c_{j}\right\} ,-\dfrac{H_{O}^{T}}{\lambda}}\right\} $
by Theorem~\ref{main theorem value function}. Therefore, $\mathsf{v}\left(O\right)=v_{i}\left(O\right)+c_{i}>v_{i}\left(O\right)$,
which is a contradiction with~\eqref{ineq: contradiction}. 
\end{proof}
\begin{lem}
\label{super_property 2}The function $v_{i}$ is a viscosity super-solution of~\eqref{eq: Hamilton Jacobi entry} at $O$.
\end{lem}
\begin{proof}[Proof of Lemma~\ref{super_property 2}] We adapt the proof of  Oudet~\cite{Oudet2014} and start by assuming
that
\[
v_{i}\left(O\right)<{\displaystyle \min\left\{ \min_{j\ne i}\left\{ v_{j}\left(O\right)+c_{j}\right\} ,-\dfrac{H_{O}^{T}}{\lambda}\right\} }.
\]
We need to prove that
\[
\lambda v_{i}\left(O\right)+H_{i}^{+}\left(O,\dfrac{d v_{i}}{d x_{i}}\left(O\right)\right)\ge0
\]
in the viscosity sense.
Let $\varphi\in C^{1}\left(\Gamma_{i}\right)$ be such that
\begin{equation}
0=v_{i}\left(O\right)-\varphi\left(O\right)\le v_{i}\left(x\right)-\varphi\left(x\right)\quad\mbox{for all } x\in \Gamma_{i},\label{property test function super solution}
\end{equation}
and $\left\{ x_{\varepsilon}\right\} \subset \Gamma_{i}\backslash\left\{ O\right\} $
be any sequence such that $x_{\varepsilon}$ tends to $O$ when $\varepsilon$
tends to $0$. From the dynamic programming principle and Lemma~\ref{super_property 1},
there exists $\bar{\tau}$ such that for any $\varepsilon>0$, there
exists $\left(y_{\varepsilon},\alpha_{\varepsilon}\right):=\left(y_{x_{\varepsilon},\alpha_{\varepsilon}},\alpha_{\varepsilon}\right)\in\mathcal{T}_{x_{\varepsilon}}$
such that $y_{\varepsilon}\left(\tau\right)\in \Gamma_{i}\backslash\left\{ O\right\} $
for any $\tau\in\left[0,\bar{\tau}\right]$ and 
\[
v_{i}\left(x_{\varepsilon}\right)+\varepsilon\ge\int_{0}^{\tau}\ell_{i}\left(y_{\varepsilon}\left(s\right),\alpha_{\varepsilon}\left(s\right)\right)e^{-\lambda s}ds+e^{-\lambda\tau}v_{i}\left(y_{\varepsilon}\left(\tau\right)\right).
\]
Then, according to~\eqref{property test function super solution}
\begin{eqnarray}
v_{i}\left(x_{\varepsilon}\right)-v_{i}\left(O\right)+\varepsilon & \ge & \int_{0}^{\tau}\ell_{i}\left(y_{\varepsilon}\left(s\right),\alpha_{\varepsilon}\left(s\right)\right)e^{-\lambda s}ds+e^{-\lambda\tau}\left[\varphi\left(y_{\varepsilon}\left(\tau\right)\right)-\varphi\left(O\right)\right]\nonumber \\
 &  & -v_{i}\left(O\right)\left(1-e^{-\lambda\tau}\right).\label{ineq: DDP (1)}
\end{eqnarray}
Next,
\[
\begin{cases}
\displaystyle \int_{0}^{\tau}\ell_{i}\left(y_{\varepsilon}\left(s\right),\alpha_{\varepsilon}\left(s\right)\right)e^{-\lambda s}ds=\int_{0}^{\tau}\ell_{i}\left(y_{\varepsilon}\left(s\right),\alpha_{\varepsilon}\left(s\right)\right)ds+o\left(\tau\right),\\
\displaystyle \left[\varphi\left(y_{\varepsilon}\left(\tau\right)\right)-\varphi\left(O\right)\right]e^{-\lambda\tau}=\varphi\left(y_{\varepsilon}\left(\tau\right)\right)-\varphi\left(O\right)+\tau o_{\varepsilon}\left(1\right)+o\left(\tau\right),
\end{cases}
\]
and 
\[
\begin{cases}
\displaystyle
v_{i}\left(x_{\varepsilon}\right)-v_{i}\left(O\right) & =o_{\varepsilon}\left(1\right),\\
\displaystyle v_{i}\left(O\right)\left(1-e^{-\lambda\tau}\right) & =o\left(\tau\right)+\tau\lambda v_{i}\left(O\right),
\end{cases}
\]
where the notation $o_{\varepsilon}\left(1\right)$ is used for a quantity
which is independent on $\tau$ and tends to $0$ as
$\varepsilon$ tends to $0$. For $k\in\mathbb{N}^{\star}$ the
notation $o(\tau^{k})$
is used for a quantity that is independent on $\varepsilon$ and such that
$\dfrac{o(\tau^{k})}{\tau^{k}}\rightarrow0$ as $\tau\rightarrow0$.
Finally, $\mathcal{O}(\tau^k)$ stands for a quantity independent on $\varepsilon$
 such that $\dfrac{\mathcal{O}(\tau^{k})}{\tau^{k}}$ 
 remains bounded as $\tau\rightarrow0$. 
From~\eqref{ineq: DDP (1)}, we obtain that
\begin{equation}
\tau\lambda v_{i}\left(O\right)\ge\int_{0}^{\tau}\ell_{i}\left(y_{\varepsilon}\left(s\right),\alpha_{\varepsilon}\left(s\right)\right)ds+\varphi\left(y_{\varepsilon}\left(\tau\right)\right)-\varphi\left(O\right)+\tau o_{\varepsilon}\left(1\right)+o\left(\tau\right)+o_{\varepsilon}\left(1\right).\label{ineq: DPP (2)}
\end{equation}
Since $y_{\varepsilon}\left(\tau\right)\in \Gamma_{i}$ for all $\varepsilon$,
one has
\[
\varphi\left(y_{\varepsilon}\left(\tau\right)\right)-\varphi\left(x_{\varepsilon}\right)=\int_{0}^{\tau}\dfrac{d\varphi}{d x_{i}}\left(y_{\varepsilon}\left(s\right)\right)\dot{y}_{\varepsilon}\left(s\right)ds=\int_{0}^{\tau}\dfrac{d\varphi}{d x_{i}}\left(y_{\varepsilon}\left(s\right)\right)f_{i}\left(y_{\varepsilon}\left(s\right),\alpha_{\varepsilon}\left(s\right)\right)ds.
\]
Hence, from~\eqref{ineq: DPP (2)}
\begin{equation}
\begin{array}{ccc}
\tau\lambda v_{i}\left(O\right)-{\displaystyle \int_{0}^{\tau}}\left[\ell_{i}\left(y_{\varepsilon}\left(s\right),\alpha_{\varepsilon}\left(s\right)\right)+\dfrac{d\varphi}{d x_{i}}\left(y_{\varepsilon}\left(s\right)\right)f_{i}\left(y_{\varepsilon}\left(s\right),\alpha_{\varepsilon}\left(s\right)\right)\right]ds & \ge & \tau o_{\varepsilon}\left(1\right)+o\left(\tau\right)+o_{\varepsilon}\left(1\right).\end{array}\label{ineq: DPP (3)}
\end{equation}
Moreover,
$\varphi\left(x_{\varepsilon}\right)-\varphi\left(O\right)=o_{\varepsilon}\left(1\right)$
and that $\dfrac{d\varphi}{d x_{i}}\left(y_{\varepsilon}\left(s\right)\right)=\dfrac{d\varphi}{d x_{i}}\left(O\right)+o_{\varepsilon}\left(1\right)+\mathcal{O}\left(s\right)$.
Thus 
\begin{equation}
\begin{array}{ccc}
\lambda v_{i}\left(O\right)-\dfrac{1}{\tau}{\displaystyle \int_{0}^{\tau}}\left[\ell_{i}\left(y_{\varepsilon}\left(s\right),\alpha_{\varepsilon}\left(s\right)\right)+\dfrac{d\varphi}{d x_{i}}\left(O\right)f_{i}\left(y_{\varepsilon}\left(s\right),\alpha_{\varepsilon}\left(s\right)\right)\right]ds & \ge & o_{\varepsilon}\left(1\right)+\dfrac{o\left(\tau\right)}{\tau}+\dfrac{o_{\varepsilon}\left(1\right)}{\tau}.\end{array}\label{ineq: DPP (4)}
\end{equation}

Let $\varepsilon_{n}\rightarrow0$ as $n\rightarrow\infty$ and $\tau_{m}\rightarrow0$
as $m\rightarrow\infty$  such that 
\[
{\displaystyle \left(a_{mn},b_{mn}\right):=\left(\dfrac{1}{\tau_{m}}\int_{0}^{\tau_{m}}f_{i}\left(y_{\varepsilon_{n}}\left(s\right),\alpha_{\varepsilon_{n}}\left(s\right)\right)e_{i}ds,\dfrac{1}{\tau_{m}}\int_{0}^{\tau_{m}}\ell_{i}\left(y_{\varepsilon_{n}}\left(s\right),\alpha_{\varepsilon_{n}}\left(s\right)\right)ds\right)}\longrightarrow\left(a,b\right)\in\mathbb{R}e_{i}\times\mathbb{R}
\]
as $n,m\rightarrow\infty$. By $\left[H1\right]$ and $\left[H2\right]$
\[
\begin{cases}
f_{i}\left(y_{\varepsilon_{n}}\left(s\right),\alpha_{\varepsilon_{n}}\left(s\right)\right)e_{i} & =f_{i}\left(O,\alpha_{\varepsilon_{n}}\left(s\right)\right)+L\left|y_{\varepsilon_{n}}\left(s\right)\right|=f_{i}\left(O,\alpha_{\varepsilon_{n}}\left(s\right)\right)e_{i}+o_{n}\left(1\right)+o_{m}\left(1\right),\\
\ell_{i}\left(y_{\varepsilon_{n}}\left(s\right),\alpha_{\varepsilon_{n}}\left(s\right)\right)e_{i} & =\ell_{i}\left(O,\alpha_{\varepsilon_{n}}\left(s\right)\right)+\omega\left(\left|y_{\varepsilon_{n}}\left(s\right)\right|\right)=\ell_{i}\left(O,\alpha_{\varepsilon_{n}}\left(s\right)\right)e_{i}+o_{n}\left(1\right)+o_{m}\left(1\right).
\end{cases}
\]
It follows that
\begin{align*}
\left(a_{mn},b_{mn}\right) & ={\displaystyle \left(\dfrac{1}{\tau_{m}}\int_{0}^{\tau_{m}}f_{i}\left(O,\alpha_{\varepsilon_{n}}\left(s\right)\right)e_{i}ds,\dfrac{1}{\tau_{m}}\int_{0}^{\tau_{m}}\ell_{i}\left(O,\alpha_{\varepsilon_{n}}\left(s\right)\right)ds\right)+o_{n}\left(1\right)+o_{m}\left(1\right)}\\
 & \in\text{FL}_{i}\left(O\right)+o_{n}\left(1\right)+o_{m}\left(1\right),
\end{align*}
since $\text{FL}_{i}\left(O\right)$ is closed and convex.
Sending
$n,m\rightarrow\infty$, we obtain $\left(a,b\right)\in\text{FL}_{i}\left(O\right)$
so there exists $\overline{a}\in A_{i}$ such that
\begin{equation}
\lim_{m,n\rightarrow\infty}\left(\dfrac{1}{\tau_{m}}\int_{0}^{\tau_{m}}f_{i}\left(y_{\varepsilon_{n}}\left(s\right),\alpha_{\varepsilon_{n}}\left(s\right)\right)e_{i}ds,\dfrac{1}{\tau_{m}}\int_{0}^{\tau_{m}}\ell_{i}\left(y_{\varepsilon_{n}}\left(s\right),\alpha_{\varepsilon_{n}}\left(s\right)\right)ds\right)=\left(f_{i}\left(O,\overline{a}\right)e_{i},\ell_{i}\left(O,\overline{a}\right)\right).\label{eq: DPP(5)}
\end{equation}
On the other hand, from Lemma~\ref{super_property 1}, $y_{\varepsilon_{n}}\left(s\right)\in \Gamma_{i}\backslash\left\{ O\right\} $ for all $s\in\left[0,\tau_{m}\right]$. This yields
\[
y_{\varepsilon_{n}}\left(\tau_{m}\right)=\left[\int_{0}^{\tau_{n}}f_{i}\left(y_{\varepsilon_{n}}\left(s\right),\alpha_{\varepsilon_{n}}\left(s\right)\right)ds\right]e_{i}+x_{\varepsilon_{n}}.
\]
Since $\left|y_{\varepsilon_{n}}\left(\tau_{m}\right)\right|>0$, then
\[
\dfrac{1}{\tau_{m}}\int_{0}^{\tau_{m}}f_{i}\left(y_{\varepsilon_{n}}\left(s\right),\alpha_{\varepsilon_{n}}\left(s\right)\right)ds\ge-\dfrac{\left|x_{\varepsilon_{n}}\right|}{\tau_{m}}.
\]
Let $\varepsilon_{n}$ tend to $0$, then let $\tau_{m}$ tend to
$0$, one gets $f_{i}\left(O,\overline{a}\right)\ge0$, so $\overline{a}\in A_{i}^{+}$.
Hence, from~\eqref{ineq: DPP (4)} and~\eqref{eq: DPP(5)}, replacing
$\varepsilon$ by $\varepsilon_{n}$ and $\tau$ by $\tau_{m}$, let
$\varepsilon_{n}$ tend to $0$, then let $\tau_{m}$ tend to $0$,
we finally obtain
\[
\lambda v_{i}\left(O\right)+\max_{a\in A_{i}^{+}}\left\{ -f_{i}\left(O,a\right)\dfrac{d\varphi}{d x_{i}}\left(O\right)-\ell_{i}\left(O,a\right)\right\} \ge\lambda v_{i}\left(O\right)+\left[-f_{i}\left(O,\overline{a}\right)\dfrac{d\varphi}{d x_{i}}\left(O\right)-\ell_{i}\left(O,\overline{a}\right)\right]\ge0.
\]
\end{proof}

\section{Comparison Principle and Uniqueness}
\label{sec:comp-princ-uniq}
Inspired by~\cite{BBC2013,BBC2014},
we begin by proving some properties of sub and super
viscosity solutions of~\eqref{eq: Hamilton Jacobi entry}. The 
following three lemmas are reminiscent of Lemma 3.4, Theorem 3.1 and Lemma
3.5 in~\cite{AOT2015}.
\begin{lem}
\label{lem: Super optimality 1} Let $w=\left(w_{1},\ldots,w_{N}\right)$
be a viscosity super-solution of~\eqref{eq: Hamilton Jacobi entry}.
Let $x\in \Gamma_{i}\backslash\left\{ O\right\} $ and assume that
\begin{eqnarray}
w_{i}\left(O\right) & < & \min\left\{ \min_{j\ne i}\left\{ w_{j}\left(O\right)+c_{j}\right\} ,-\dfrac{H_{O}^{T}}{\lambda}\right\} .\label{ineq: CP DPP_enter}
\end{eqnarray}
Then for all $t>0$,
\begin{eqnarray*}
w_{i}\left(x\right) & \ge & \inf_{\alpha_{i}\left(\cdot\right),\theta_{i}}\left(\int_{0}^{t\wedge\theta_{i}}\ell_{i}\left(y_{x}^{i}\left(s\right),\alpha_{i}\left(s\right)\right)e^{-\lambda s}ds+w_{i}\left(y_{x}^{i}\left(t\wedge\theta_{i}\right)\right)e^{-\lambda\left(t\wedge\theta_{i}\right)}\right),
\end{eqnarray*}
where $\alpha_{i}\in L^{\infty}\left(0,\infty;A_{i}\right)$, $y_{x}^{i}$
is the solution of $y_{x}^{i}\left(t\right)=x+\left[\int_{0}^{t}f_{i}\left(y_{x}^{i}\left(s\right),\alpha_{i}\left(s\right)\right)ds\right]e_{i}$
and $\theta_{i}$ satisfies $y_{x}^{i}\left(\theta_{i}\right)=0$
and $\theta_{i}$ lies in $\left[\tau_{i},\overline{\tau_{i}}\right]$,
where $\tau_{i}$ is the exit time of $y_{x}^{i}$ from $\Gamma_{i}\backslash\left\{ O\right\} $
and $\overline{\tau_{i}}$ is the exit time of $y_{x}^{i}$ from $\Gamma_{i}$.
\end{lem}
\begin{proof}[Proof of Lemma~\ref{lem: Super optimality 1}]
According to~\eqref{ineq: CP DPP_enter}, the function $w_{i}$ is a
viscosity super-solution of the following problem in $\Gamma_{i}$
\begin{equation}
\begin{cases}
\lambda w_{i}\left(x\right)+H_{i}\left(x,\dfrac{d w_{i}}{d x_{i}}\left(x\right)\right) & =0\quad\mbox{if }x\in \Gamma_{i}\backslash\left\{ O\right\} ,\\
\lambda w_{i}\left(O\right)+H_{i}^{+}\left(O,\dfrac{d w_{i}}{dx_{i}}\left(O\right)\right) & =0\quad\mbox{if }x=O.
\end{cases}\label{eq:H-J equation with H+}
\end{equation}
 Hence, we can apply the result  in~\cite[Lemma 3.4]{AOT2015}. We refer to~\cite{BBC2013} for a detailed proof. The main point of that
proof uses the results of Blanc~\cite{Blanc1997,Blanc2001} on minimal
super-solutions of exit time control problems.
\end{proof}
\begin{lem}[Super-optimality]
\label{lem: Super optimality 2} Under assumption $\left[H\right]$,
let $w=\left(w_{1},\ldots,w_{N}\right)$ be a viscosity super-solution
of~\eqref{eq: Hamilton Jacobi entry} that satisfies~\eqref{ineq: CP DPP_enter};
then there exists a sequence $\left\{ \eta_{k}\right\} _{k\in\mathbb{N}}$
of strictly positive real numbers such that $\lim_{k\rightarrow\infty}\eta_{k}=\eta>0$
and a sequence $x_{k}\in \Gamma_{i}\backslash\left\{ O\right\} $ such
that $\lim_{k\rightarrow\infty}x_{k}=O,\lim_{k\rightarrow\infty}w_{i}\left(x_{k}\right)=w_{i}\left(O\right)$
and for each $k$, there exists a control law $\alpha_{i}^{k}$ such
that the corresponding trajectory $y_{x_{k}}\left(s\right)\in \Gamma_{i}$
for all $s\in\left[0,\eta_{k}\right]$ and
\[
w_{i}\left(x_{k}\right)\ge\int_{0}^{\eta_{k}}\ell_{i}\left(y_{x_{k}}\left(s\right),\alpha_{i}^{k}\left(s\right)\right)e^{-\lambda s}ds+w_{i}\left(y_{x_{k}}\left(\eta_{k}\right)\right)e^{-\lambda\eta_{k}}.
\]
\end{lem}
\begin{proof}[Proof of Lemma~\ref{lem: Super optimality 2}]
According to~\eqref{ineq: CP DPP_enter} $\widehat{w}_{i}\left(O\right)<-\dfrac{H_{O}^{T}}{\lambda}$.
Hence, this proof is complete by applying the proof of 
in~\cite[Theorem 3.1]{AOT2015}.
\end{proof}
\begin{lem}\label{lem:test function at O}
Under assumption $\left[H\right]$, let $u=\left(u_{1},\ldots,u_{N}\right)$
be a viscosity sub-solution of~\eqref{eq: Hamilton Jacobi entry}.
Then $u_{i}$ is Lipschitz continuous in $B\left(O,r\right)\cap \Gamma_{i}$. 
Therefore, there exists
a test function $\varphi_{i}\in C^{1}\left({\Gamma_{i}}\right)$
which touches $u_{i}$ from above at $O$.
\end{lem}
\begin{proof}[Proof of Lemma~\ref{lem:test function at O}]
Since $u$ is a viscosity sub-solution of~\eqref{eq: Hamilton Jacobi entry},
$u_{i}$ is a viscosity sub-solution of~\eqref{eq:H-J equation with H+}. 
Recal that $H_{i}\left(x,\cdot\right)$ is coercive
for any $x\in \Gamma_{i}\cap B\left(O,r\right)$, we can apply the proof
 in~\cite[Lemma 3.2]{AOT2015}, which is based on arguments due to Ishii and contained in~\cite{Ishii2013}.
\end{proof}
\begin{lem}[Sub-optimality]
\label{lem: Sub optimality} Under assumption $\left[H\right]$,
let $u=\left(u_{1},\ldots,u_{N}\right)$ be a viscosity sub-solution
of~\eqref{eq: Hamilton Jacobi entry}. Consider $i=\overline{1,N} ,x\in \Gamma_{i}\backslash\left\{ O\right\} $
and $\alpha_{i}\in L^{\infty}\left(0,\infty;A_{i}\right)$. Let $T>0$
be such that $y_{x}\left(t\right)=x+\left[\int_{0}^{t}f_{i}\left(y_{x}\left(s\right),\alpha_{i}\left(s\right)\right)ds\right]e_{i}$
belongs to $\Gamma_{i}$ for any $t\in\left[0,T\right]$, then
\[
u_{i}\left(x\right)\le\int_{0}^{T}\ell_{i}\left(y_{x}\left(s\right),\alpha_{i}\left(s\right)\right)e^{-\lambda s}ds+u_{i}\left(y_{x}\left(T\right)\right)e^{-\lambda T}.
\]
\end{lem}
\begin{proof}[Proof of Lemma~\ref{lem: Sub optimality}]
Since $u$ is a viscosity sub-solution of~\eqref{eq: Hamilton Jacobi entry},
$u_{i}$ is a viscosity sub-solution of~\eqref{eq:H-J equation with H+}. 
and satisfies $u_{i}\left(O\right)\le-\dfrac{H_{O}^{T}}{\lambda}$.
Hence, we can apply the proof in~\cite[Lemma 3.5 ]{AOT2015}.
\end{proof}
\begin{rem}
Under assumption $\left[H\right]$, Lemmas~\ref{lem: Super optimality 1},~\ref{lem: Super optimality 2},~\ref{lem:test function at O} and~\ref{lem: Sub optimality} hold for vicosity sub- and super-solution $\hat{u}$ and $\hat{w}$ repestively, of the exit cost control problem if~\eqref{ineq: CP DPP_enter} replaced by
\[
\widehat{w}_{i}\left(O\right)<\min\left\{ \min_{j\ne i}\left\{ \widehat{w}_{j}\left(O\right)\right\} +d_{i},-\dfrac{H_{O}^{T}}{\lambda}+d_{i}\right\} .
\]
\end{rem}
\begin{thm}[Comparison Principle]
\label{thm: Comparison Principle} Under assumption $\left[H\right]$,
let $u$
be a bounded viscosity sub-solution of~\eqref{eq: Hamilton Jacobi entry} and $w$ be a bounded viscosity super-solution of~\eqref{eq: Hamilton Jacobi entry}; then $u\le w$ in $\mathcal{G}$,
componentwise. This theorem also holds for viscosity sub- and super-solution $\widehat{u}$ and $\widehat{w}$, respectively, of the exit cost control problem~\eqref{eq: Hamilton Jacobi exit}.
\end{thm}
We give two proofs of Theorem~\ref{thm: Comparison Principle}. The first one is inspired by~\cite{AOT2015} and uses the 
previously stated lemmas. The second one uses the  elegant arguments proposed in~\cite{LS2016}.
\begin{proof}[{{Proof of Theorem~\ref{thm: Comparison Principle} inspired by~\cite{AOT2015} }}]
 We focus on $u$ and $w$, the arguments used for the comparison of $\widehat{u}$ and $\widehat{w}$
are totally similar. Suppose by contradiction that there exists $x\in \Gamma_{i}$
such that $u_{i}\left(x\right)-w_{i}\left(x\right)>0$. By classical
comparison arguments for the boundary value problem, see~\cite{Barles2013}, $\sup_{\partial \Gamma_{i}}\left\{ u_{i}-v_{i}\right\} ^{+}\ge\sup_{\Gamma_{i}}\left\{ u_{i}-v_{i}\right\} ^{+}$, so we have
\[
u_{i}\left(O\right)-w_{i}\left(O\right)=\max_{x\in \Gamma_{i}}\left\{ u_{i}\left(x\right)-w_{i}\left(x\right)\right\} >0.
\]
By definition of viscosity sub-solution
\begin{equation}
\lambda u_{i}\left(O\right)+H_{O}^{T}\le0.\label{ineq: CP HOT}
\end{equation}
This implies $\lambda w_{i}\left(O\right)+H_{O}^{T}<0$. We now consider
the two following cases.

\begin{description}
\item{\emph{Case 1:}} If $w_{i}\left(O\right)<\min_{j\ne i}\left\{ w_{j}\left(O\right)+c_{j}\right\} $,
from Lemma~\ref{lem: Super optimality 2} (using the same notations),
\[
w_{i}\left(x_{k}\right)\ge\int_{0}^{\eta_{k}}\ell_{i}\left(y_{x_{k}}\left(s\right),\alpha_{i}^{k}\left(s\right)\right)e^{-\lambda s}ds+w_{i}\left(y_{x_{k}}\left(\eta_{k}\right)\right)e^{-\lambda\eta_{k}}.
\]
Moreover, according to Lemma~\ref{lem: Sub optimality}, we also have
\[
u_{i}\left(x_{k}\right)\le\int_{0}^{\eta_{k}}\ell_{i}\left(y_{x_{k}}\left(s\right),\alpha_{i}^{k}\left(s\right)\right)e^{-\lambda s}ds+u_{i}\left(y_{x_{k}}\left(\eta_{k}\right)\right)e^{-\lambda\eta_{k}}.
\]
This yields
\[
u_{i}\left(x_{k}\right)-w_{i}\left(x_{k}\right)\le\left[u_{i}\left(y_{x_{k}}\left(\eta_{k}\right)\right)-w_{i}\left(y_{x_{k}}\left(\eta_{k}\right)\right)\right]e^{-\lambda\eta_{k}}\le\left[u_{i}\left(O\right)-w_{i}\left(O\right)\right]e^{-\lambda\eta_{k}}.
\]
By letting $k$ tend to $\infty$, one gets
\[
u_{i}\left(O\right)-w_{i}\left(O\right)\le\left[u_{i}\left(O\right)-w_{i}\left(O\right)\right]e^{-\lambda\eta}.
\]
This implies that $u_{i}\left(O\right)-w_{i}\left(O\right)\le0$ and leads
to a contradiction.
\item{\emph{Case 2:}} If $w_{i}\left(O\right)\ge\min_{j\ne i}\left\{ w_{j}\left(O\right)+c_{j}\right\} $,
then there exists $j_{0}\ne i$ such that
\[
w_{j_{0}}\left(O\right)+c_{j_{0}}=\min_{j=\overline{1,N}}\left\{ w_{j}\left(O\right)+c_{j}\right\} =\min_{j\ne i}\left\{ w_{j}\left(O\right)+c_{j}\right\} \le w_{i}\left(O\right),
\]
because $c_i>0$. Since $c_{j_{0}}$ is positive
\begin{equation}
w_{j_{0}}\left(O\right)<\min_{j\ne j_{0}}\left\{ w_{j}\left(O\right)+c_{j}\right\}.\label{ineq: comparison 1}
\end{equation}

Next, by Lemma~\ref{lem:test function at O}, there exists a test function $\varphi_{i}$  in $C^{1}\left({J}_{i}\right)$
that touches $u_i$ from above at $O$, it yields 
\[
\lambda u_{i}\left(O\right)-\lambda\min_{j\ne i}\left\{ u_{j}\left(O\right)+c_{j}\right\} \le\lambda u_{i}\left(O\right)+\max\left\{ -\lambda\min_{j\ne i}\left\{ u_{j}\left(O\right)+c_{j}\right\} ,H_{i}^{+}\left(O,\dfrac{d\varphi_{i}}{dx_{i}}\left(O\right)\right),H_{0}^{T}\right\} \le0.
\]
Therefore
\[
w_{j_{0}}\left(O\right)+c_{j_{0}}\le w_{i}\left(O\right)<u_{i}\left(O\right)\le\min_{j\ne i}\left\{ u_{j}\left(O\right)+c_{j}\right\} \le u_{j_{0}}\left(O\right)+c_{j_{0}}.
\]
Thus 
\begin{equation}
w_{j_{0}}\left(O\right)<u_{j_{0}}\left(O\right).\label{ineq: comparison}
\end{equation}
Replacing index $i$ by $j_{0}$ in~\eqref{ineq: CP HOT},
we get 
\begin{equation}
\lambda w_{j_{0}}\left(O\right)+H_{O}^{T}<0.\label{ineq: comparison 2}
\end{equation}
By~\eqref{ineq: comparison 1} and~\eqref{ineq: comparison 2},~\eqref{ineq: CP DPP_enter} holds true. Repeating the proof of \emph{Case 1} with $j_{0}$, we reach a contradiction with~\eqref{ineq: comparison}. It ends the proof. 
\end{description}
\end{proof}
The comparison principle can also be obtained alternatively, using 
the arguments which were very recently proposed by Lions and Souganidis  in~\cite{LS2016}.
This new proof is self-combined and the arguments do not rely at all on optimal control theory, but are deeply connected to the ideas used by  Soner~\cite{Soner1986a,Soner1986b} and Capuzzo-Dolcetta and Lions~\cite{DL1990} for proving comparison principles for state-constrained Hamilton-Jacobi equations
\begin{proof}[Proof of Theorem~\ref{thm: Comparison Principle} inspired by~\cite{LS2016}]
We start as in first proof. We argue by contradiction without loss of generality, assuming that there exists $i$ such that
\[
u_{i}\left(O\right)-w_{i}\left(O\right)=\max_{\Gamma_{i}}\left\{ u_{i}\left(x\right)-w_{i}\left(x\right)\right\} >0.
\]
Therefore $w_{i}\left(O\right)<-\dfrac{H_{O}^{T}}{\lambda}$. We now consider
the two following cases.
\begin{description}
\item{\emph{Case 1:}}
 If $w_{i}\left(O\right)<\min_{j\ne i}\left\{ w_{j}\left(O\right)+c_{j}\right\} $,
then $w_{i}$ is a viscosity super-solution of~\eqref{eq:H-J equation with H+}. Recall that by Lemma~\ref{lem:test function at O}, there exists a positive number $L$ such that for $i=\overline{1,N}$, $u_i$ is Lipschitz continuous with Lipschitz constant $L$ in $\Gamma_{i}\cap B(0,r)$.
We consider the function
\begin{align*}
\Psi_{i,\varepsilon}:\Gamma_{i}\times \Gamma_{i} & \longrightarrow\mathbb{R}\\
\left(x,y\right) & \longrightarrow u_{i}\left(x\right)-w_{i}\left(y\right)-\dfrac{1}{2\varepsilon}\left[-\left|x\right|+\left|y\right|+\delta\left(\varepsilon\right)\right]^{2}-\gamma\left(\left|x\right|+\left|y\right|\right),
\end{align*}
where $\delta\left(\varepsilon\right)=\left(L+1\right)\varepsilon$ and
$\gamma\in\left(0,\dfrac{1}{2}\right)$. It is clear that  $\Psi_{i,\varepsilon}$
attains its maximum $M_{\varepsilon,\gamma}$ at $\left(x_{\varepsilon,\gamma},y_{\varepsilon,\gamma}\right)\in \Gamma_{i}\times \Gamma_{i}$.
By classical techniques,  we check that $x_{\varepsilon,\gamma},y_{\varepsilon,\gamma}\rightarrow O$
and that $\dfrac{\left(x_{\varepsilon,\gamma}-y_{\varepsilon,\gamma}\right)^{2}}{\varepsilon}\rightarrow0$
as $\varepsilon\rightarrow0$. Indeed, one has
\begin{eqnarray}\notag
 & & u_{i}\left(x_{\varepsilon,\gamma}\right)-w_{i}\left(y_{\varepsilon,\gamma}\right)-\dfrac{\left[-\left|x_{\varepsilon,\gamma}\right|+\left|y_{\varepsilon,\gamma}\right|+\delta\left(\varepsilon\right)\right]^{2}}{2\varepsilon}-\gamma\left(\left|x_{\varepsilon,\gamma}\right|+\left|y_{\varepsilon,\gamma}\right|\right)\nonumber \\
&\ge & \max_{\Gamma_{i}}\left\{ u_{i}\left(x\right)-w_{i}\left(x\right)-2\gamma\left|x\right|\right\} -\dfrac{\delta^{2}\left(\varepsilon\right)}{2\varepsilon}\label{ineq: Lions-Souganidis classical 1}\\
&\ge & u_{i}\left(O\right)-w_{i}\left(O\right)-\dfrac{\left(L+1\right)^{2}}{2}\varepsilon.\label{ineq: Lions-Souganidis classical 2}
\end{eqnarray}
Since $u_{i}\left(O\right)-v_{i}\left(O\right)>0$, the term in~\eqref{ineq: Lions-Souganidis classical 2} is
positive when $\varepsilon$ is small enough.
We also deduce from the above inequality and from the boundedness of $u_{i}$ and $w_{i}$ that, maybe after the extraction of a subsequence,  $x_{\varepsilon,\gamma},y_{\varepsilon,\gamma}\rightarrow x_{\gamma}$
as $\varepsilon\rightarrow0$, for some  $x_{\gamma}\in \Gamma_{i}$.
 From~\eqref{ineq: Lions-Souganidis classical 1},
\[
u_{i}\left(x_{\varepsilon,\gamma}\right)-w_{i}\left(y_{\varepsilon,\gamma}\right)-\dfrac{\left(\left|x_{\varepsilon,\gamma}\right|-\left|y_{\varepsilon,\gamma}\right|\right)^{2}}{2\varepsilon}-\dfrac{\left(-\left|x_{\varepsilon,\gamma}\right|+\left|y_{\varepsilon,\gamma}\right|\right)\delta\left(\varepsilon\right)}{\varepsilon}\ge\max_{\Gamma_{i}}\left\{ u_{i}\left(x\right)-w_{i}\left(x\right)-2\gamma\left|x\right|\right\} .
\]
Taking the $\limsup$ on both sides of this inequality when $\varepsilon\rightarrow0$,
\begin{align*}
u_{i}\left(x_{\gamma}\right)-w_{i}\left(x_{\gamma}\right)-2\gamma\left|x_{\gamma}\right| & \ge\max_{\Gamma_{i}}\left\{ u_{i}\left(x\right)-w_{i}\left(x\right)-2\gamma\left|x\right|\right\} +\limsup_{\varepsilon\rightarrow0}\dfrac{\left(\left|x_{\varepsilon,\gamma}\right|-\left|y_{\varepsilon,\gamma}\right|\right)^{2}}{2\varepsilon}\\
 & \ge u_{i}\left(O\right)-w_{i}\left(O\right)+\limsup_{\varepsilon\rightarrow0}\dfrac{\left(\left|x_{\varepsilon,\gamma}\right|-\left|y_{\varepsilon,\gamma}\right|\right)^{2}}{2\varepsilon}\\
 & \ge u_{i}\left(O\right)-w_{i}\left(O\right)+\liminf_{\varepsilon\rightarrow0}\dfrac{\left(\left|x_{\varepsilon,\gamma}\right|-\left|y_{\varepsilon,\gamma}\right|\right)^{2}}{2\varepsilon}\\
 & \ge u_{i}\left(O\right)-w_{i}\left(O\right).
\end{align*}
Recalling that $u_{i}\left(O\right)-w_{i}\left(O\right)=\max_{\Gamma_{i}}\left(u_{i}-w_{i}\right)$,
we obtain from the inequalities above that  $x_{\gamma}=O$ and that
\begin{equation}\label{eq: distance x,y and epsilon}
 \lim_{\varepsilon\rightarrow0}\dfrac{\left(\left|x_{\varepsilon,\gamma}\right|-\left|y_{\varepsilon,\gamma}\right|\right)^{2}}{2\varepsilon}=0.
\end{equation}
We claim that if $\varepsilon>0$, then $x_{\varepsilon,\gamma}\ne O$. 
Indeed, assume by contradiction that $x_{\varepsilon,\gamma}=O$:
\begin{enumerate}
\item if $y_{\varepsilon,\gamma}>0$, then
\[
M_{\varepsilon,\gamma}=u_{i}\left(O\right)-w_{i}\left(y_{\varepsilon,\gamma}\right)-\dfrac{1}{2\varepsilon}\left[\left|y_{\varepsilon,\gamma}\right|+\delta\left(\varepsilon\right)\right]^{2}-\gamma\left|y_{\varepsilon,\gamma}\right|\ge u_{i}\left(y_{\varepsilon,\gamma}\right)-w_{i}\left(y_{\varepsilon,\gamma}\right)-\dfrac{\delta^{2}\left(\varepsilon\right)}{2\varepsilon}-2\gamma\left|y_{\varepsilon,\gamma}\right|.
\]
Since $u_{i}$ is Lipschitz continuous in $B\left(O,r\right)\cap \Gamma_{i}$,  we see that for $\varepsilon$ small enough
\[
L\left|y_{\varepsilon.\gamma}\right|\ge u_{i}\left(O\right)-u_{i}\left(y_{\varepsilon,\gamma}\right)\ge\dfrac{\left|y_{\varepsilon,\gamma}\right|^{2}}{2\varepsilon}+\dfrac{\left|y_{\varepsilon,\gamma}\right|\delta\left(\varepsilon\right)}{\varepsilon}-\gamma\left|y_{\varepsilon,\gamma}\right|\ge\dfrac{\left|y_{\varepsilon,\gamma}\right|\delta\left(\varepsilon\right)}{\varepsilon}-\gamma\left|y_{\varepsilon,\gamma}\right|.
\]
Therefore, if $y_{\varepsilon,\gamma}\not=O$, then $L\ge L+1-\gamma$ which
gives a contradiction since $\gamma\in\left(0,\dfrac{1}{2}\right)$.
\item 
Otherwise, if $y_{\varepsilon,\gamma}=O$, then 
\[
M_{\varepsilon,\gamma}=u_{i}\left(O\right)-w_{i}\left(O\right)-\dfrac{\delta^{2}\left(\varepsilon\right)}{2\varepsilon}
\ge u_{i}\left(\varepsilon {e}_{i}\right)-w_{i}\left(O\right)-\dfrac{1}{2\varepsilon}\left[-\varepsilon+\delta\left(\varepsilon\right)\right]^{2}-\gamma\varepsilon.
\]
Since $u_{i}$ is Lipschitz continuous in $B\left(O,r\right)\cap \Gamma_{i}$, we see that for $\varepsilon$ small enough, 
\[
L\varepsilon\ge u_{i}\left(O\right)-u_{i}\left(\varepsilon e_{i}\right)\ge\dfrac{\left|y_{\varepsilon.\gamma}\right|^{2}}{2\varepsilon}+\dfrac{\left|y_{\varepsilon.\gamma}\right|\delta\left(\varepsilon\right)}{\varepsilon}-2\gamma\left|y_{\varepsilon.\gamma}\right|\ge\dfrac{\left|y_{\varepsilon.\gamma}\right|\delta\left(\varepsilon\right)}{\varepsilon}-2\gamma\left|y_{\varepsilon.\gamma}\right|.
\]
This implies that $L\ge-\dfrac{1}{2}+L+1-\gamma$,
which gives a contradiction since $\gamma\in\left(0,\dfrac{1}{2}\right)$.     
\end{enumerate}
Therefore the claim is proved. It follows that we can apply the viscosity inequality for $u_i$ at $x_{\varepsilon,\gamma}$. Moreover, notice that the viscosity super-solution inequality~\eqref{eq:H-J equation with H+} holds also for $y_{\varepsilon,\gamma}=0$ since $H_{i}\left(O,p\right)\le H^{+}_{i}\left(O,p\right)$ for any $p$. Therefore
\begin{align*}
u_{i}\left(x_{\varepsilon,\gamma}\right)+H_{i}\left(x_{\varepsilon,\gamma},\dfrac{-x_{\varepsilon,\gamma}+y_{\varepsilon,\gamma}+\delta\left(\varepsilon\right)}{\varepsilon}+\gamma\right) & \le0,\\
w_{i}\left(y_{\varepsilon,\gamma}\right)+H_{i}\left(y_{\varepsilon,\gamma},\dfrac{-x_{\varepsilon,\gamma}+y_{\varepsilon,\gamma}+\delta\left(\varepsilon\right)}{\varepsilon}-\gamma\right) & \ge0.
\end{align*}
Subtracting the two inequalities,
\begin{equation}\label{ineq: estimate Hamiltonian}
u_{i}\left(x_{\varepsilon,\gamma}\right)-w_{i}\left(y_{\varepsilon,\gamma}\right)\le H_{i}\left(y_{\varepsilon,\gamma},\dfrac{-x_{\varepsilon,\gamma}+y_{\varepsilon,\gamma}+\delta\left(\varepsilon\right)}{\varepsilon}+\gamma\right)-H_{i}\left(x_{\varepsilon,\gamma},\dfrac{-x_{\varepsilon,\gamma}+y_{\varepsilon,\gamma}+\delta\left(\varepsilon\right)}{\varepsilon}-\gamma\right).
\end{equation}

Using $\left[H1\right]$ and $\left[H2\right]$, it is easy to see that there exists $\overline{M}_{i}>0$
such that for any $x,y\in \Gamma_{i},p,q\in\mathbb{R}$
\begin{align*}
\left|H_{i}\left(x,p\right)-H_{i}\left(y,q\right)\right| & \le\left|H_{i}\left(x,p\right)-H_{i}\left(y,p\right)\right|+\left|H_{i}\left(y,p\right)-H_{i}\left(y,q\right)\right|\\
 & \le\overline{M}_{i}\left|x-y\right|\left(1+\left|p\right|\right)+\overline{M}_{i}\left|p-q\right|.
\end{align*}
It yields
\begin{align*}
u_{i}\left(x_{\varepsilon,\gamma}\right)-w_{i}\left(y_{\varepsilon,\gamma}\right) & \le\overline{M}_{i}\left[\left|x_{\varepsilon,\gamma}-y_{\varepsilon,\gamma}\right|\left(1+\left|\dfrac{-x_{\varepsilon,\gamma}+y_{\varepsilon,\gamma}+\delta\left(\varepsilon\right)}{\varepsilon}-\gamma\right|\right)+2\left|\gamma\right|\right]\\
 & \le\overline{M}_{i}\left[\left|x_{\varepsilon,\gamma}-y_{\varepsilon,\gamma}\right|\left(\gamma+1+\dfrac{\delta\left(\varepsilon\right)}{\varepsilon}\right)+\dfrac{\left|x_{\varepsilon,\gamma}-y_{\varepsilon,\gamma}\right|^{2}}{\varepsilon}+2\left|\gamma\right|\right].
\end{align*}
Applying~\eqref{eq: distance x,y and epsilon},  let $\varepsilon$ tend to $0$ and $\gamma$ tend to $0$, we
obtain that $u_{i}\left(O\right)-w_{i}\left(O\right)\le 0$, the desired contradiction.
\item{\emph{Case 2:}}  $w_{i}\left(O\right)\ge\min_{j\ne i}\left\{ w_{j}\left(O\right)+c_{j}\right\} =w_{j_{0}}\left(O\right)+c_{j_{0}}$.
Using the same arguments as in \emph{Case 2} of the first proof, we get
\[
w_{j_{0}}<\min\left\{ \min_{j\ne j_{0}}\left\{ w_{j}\left(O\right)+c_{j}\right\} ,-\dfrac{H_{O}^{T}}{\lambda}\right\} 
\]
and $w_{j_{0}}\left(O\right)<u_{j_{0}}\left(O\right)$. Repeating \emph{Case 1}, replacing the index $i$ by $j_{0}$, implies
that $w_{j_{0}}\left(O\right)\ge u_{j_{0}}\left(O\right)$, the desired contradiction.
\end{description}
\end{proof}
\begin{cor}[Uniqueness]
\label{cor: Uniqueness}
If $\mathsf{v}$ is the value function (with entry costs)
and  $\left(v_{1},\ldots,v_{N}\right)$ is defined by
\[
v_{i}\left(x\right)=\begin{cases}
\mathsf{v}\left(x\right) & \quad\mbox{if }x\in \Gamma_{i}\backslash\left\{ O\right\} ,\\
\lim_{\delta\rightarrow0^{+}}\mathsf{v}\left(\delta e_{i}\right) & \quad\mbox{if }x=O,
\end{cases}
\]
then  $\left(v_{1},\ldots,v_{N}\right)$ is the unique
bounded viscosity solution of~\eqref{eq: Hamilton Jacobi entry}.
\\
Similarly, if $\widehat{\mathsf{v}}$ is the value function (with exit costs)
and $\left(\widehat{v}_{1},\ldots,\widehat{v}_{N}\right)$ is defined by 
\[
\widehat{v}_{i}\left(x\right)=\begin{cases}
\widehat{\mathsf{v}}\left(x\right) & \quad\mbox{if }x\in \Gamma_{i}\backslash\left\{ O\right\} ,\\
\lim_{\delta\rightarrow0^{+}}\widehat{\mathsf{v}}\left(\delta e_{i}\right) & \quad\mbox{if }x=O,
\end{cases}
\]
then $\left(\widehat{v}_{1},\ldots,\widehat{v}_{N}\right)$ is the unique
bounded viscosity solution of~\eqref{eq: Hamilton Jacobi exit}.
\end{cor}

\begin{rem}
From Corollary \ref{cor: Uniqueness}, we see that  in order to characterize the original value function with entry costs, we need to solve first the Hamilton-Jacobi system \eqref{eq: Hamilton Jacobi entry} and find the unique viscosity solution $\left(v_{1},\ldots,v_{N}\right)$. The original value function  $\mathsf{v}$ with entry costs  satisfies
\[
\mathsf{v}\left(x\right)=\begin{cases}
v_{i}\left(x\right) & \text{if }x\in\Gamma_{i}\backslash\left\{ O\right\} ,\\
\min\left\{ \min_{i=\overline{1,N}}\left\{ v_{i}\left(O\right)+c_{i}\right\} ,-\dfrac{H_{O}^{T}}{\lambda}\right\} , & \text{if }x=O.
\end{cases}
\]
The characterization of $\mathsf{v}\left(O\right)$ follows from Theorem \ref{main theorem value function}. The characterization of the original value function with exit costs $\widehat{\mathsf{v}}$ is similar.
\end{rem}
\section{A more general optimal control problem }
\label{sec:gener-case-switch}
In what follows, we generalize the control problem studied in the previous sections
 by allowing some of the entry (or exit) costs to be zero.
The situation can be viewed as intermediary between the one studied in~\cite{AOT2015} when all the entry (or exit) costs were
 zero, and that studied above when all the entry or exit costs were positive. 
 Accordingly, every result presented below will mainly be obtained by combining the arguments proposed above
with those used in~\cite{AOT2015}. Hence, we will present the results and omit the proofs.

To be more specific, we  consider the optimal control problems  with non-negative entry cost
 $\overline{C}=\left\{ \overline{c}_{1},\ldots\overline{c}_{m},\overline{c}_{m+1},\ldots\overline{c}_{N}\right\} $
where $\overline{c}_{i}=0$ if $i\le m$ and $\overline{c}_{i}>0$
if $i>m$, keeping all the assumptions and definitions of Section \ref{sec:optim-contr-probl}
unchanged. The value function associated to $\overline{C}$ will be
denoted by $\mathsf{V}$. Similarly to Lemma~\ref{lem: extension value function},
$\mathsf{V}|_{\Gamma_{i}\backslash\left\{ O\right\} }$  is continuous and Lipschitz continuous
near $O$: therefore, it is possible to extend $\mathsf{V}|_{\Gamma_{i}\backslash\left\{ O\right\} }$
at $O$. This extension will be noted  $\mathcal{V}_{i}$. Moreover, one
can check that $\mathcal{V}_{i}\left(O\right)=\mathcal{V}_{j}\left(O\right)$ for all
$i,j\le m$, which means that $\mathsf{V}|_{\cup_{i=1}^{m}\Gamma_{i}}$ is a continuous function
which will be noted  $\mathcal{V}_{c}$ hereafter.

Combining the arguments  in~\cite{AOT2015} and in Section~\ref{sec:optim-contr-probl} leads us to the following
theorem.
\begin{thm}
The value function $\mathsf{V}$ satisfies
\[
\max_{i=\overline{m+1,N}}\left\{ \mathcal{V}_{i}\left(O\right)\right\} \le \mathsf{V}\left(O\right)=\mathcal{V}_{c}\left(O\right)\le\min\left\{ \min_{i=\overline{m+1,N}}\left\{ \mathcal{V}_{i}\left(O\right)+\overline{c}_{i}\right\} ,-\dfrac{H_{O}^{T}}{\lambda}\right\} .
\]
\end{thm}

\begin{rem}
In the case when $\overline{c}_{i}=0$ for $i=\overline{1,N} $,
 $\mathsf{V}$ is continuous on $\mathcal{G}$ and it is
exactly the value function of the problem studied in~\cite{AOT2015}.
\end{rem}
We now define a set of admissible test-function  and the Hamilton-Jacobi
equation that will characterize $\mathsf{V}$.
\begin{defn}
A function $\varphi:\left(\cup_{i=1}^{m}\Gamma_{i}\right)\times \Gamma_{m+1}\times\ldots\times \Gamma_{N}\rightarrow\mathbb{R}^{N-m+1}$ of the form  $\varphi\left(x_{c},x_{m+1},\ldots,x_{N}\right)=\left(\varphi_{c}\left(x_{c}\right),\varphi_{m+1}\left(x_{m+1}\right),\ldots,\varphi_{N}\left(x_{N}\right)\right)$
is an admissible test-function
if 
\begin{itemize}
\item $\varphi_{c}$ is continuous and for $i\le m$, $\varphi_{c}|_{\Gamma_{i}}$
belongs to $C^{1}\left(\Gamma_{i}\right)$,
\item for $i>m$, $\varphi_{i}$ belongs to $C^{1}\left(\Gamma_{i}\right)$,
\item the space of admissible test-function is noted $R\left(\mathcal{G}\right)$.
\end{itemize}
\end{defn}
\begin{defn}
\label{def: viscosity of general HJ} A function $U=\left(U_{c},U_{m+1},\ldots,U_{N}\right)$
where $U_{c}\in USC\left(\cup_{j=1}^{m}\Gamma_{j};\mathbb{R}\right),U_{i}\in USC\left(\Gamma_{i};\mathbb{R}\right)$
is called a \emph{viscosity sub-solution} of the Hamilton-Jacobi
system if  for any $\left(\varphi_{c},\varphi_{m+1},\ldots,\varphi_{N}\right)\in R\left(\mathcal{G}\right)$:
\begin{enumerate}
\item 
if  $U_{c}-\varphi_{c}$ has a local maximum  at $x_{c}\in \cup_{j=1}^{m}\Gamma_{j}$ and if
\begin{itemize}
\item    $x_{c}\in \Gamma_{j}\backslash\left\{ O\right\} $ for some $j\le m$, then
\[
\begin{array}{c}
\lambda U_{c}\left(x_{c}\right)+H_{j}\left(x,\dfrac{d\varphi_{c}}{d x_{j}}\left(x_{c}\right)\right)\le0,\end{array}
\]
\item  $x_{c}=O$, then
\[
\lambda U_{c}\left(O\right)+\max\left\{ -\lambda{\displaystyle \min_{j>m}}\left\{ U_{j}\left(O\right)+\overline{c}_{j}\right\} ,\max_{j\le m}\left\{ H_{j}^{+}\left(O,\dfrac{d\varphi_{c}}{d x_{j}^{+}}\left(O\right)\right)\right\} ,H_{O}^{T}\right\} \le0;
\]
\end{itemize}
\item  if $U_{i}-\varphi_{i}$ has a \emph{local maximum point}  at $x_{i}\in \Gamma_{i}$ for $i>m$, and if
\begin{itemize}
\item  $\ensuremath{x_{i}\in \Gamma_{i}\backslash \left\{O\right\}}$, then 
\[
\begin{array}{c}
\lambda U_{i}\left(x_{i}\right)+H_{i}\left(x,\dfrac{d\varphi_{i}}{d x_{i}}\left(x_{i}\right)\right)\le0,\end{array}
\]
\item  $x_{i}=O$, then
\[
{\displaystyle \lambda U_{i}\left(O\right)+\max\left\{ -\lambda\min_{j>m,j\ne i}\left\{ U_{j}\left(O\right)+\overline{c}_{j}\right\} ,-\lambda U_{c}\left(O\right),H_{i}^{+}\left(O,\dfrac{d\varphi_{i}}{dx_{i}}\left(O\right)\right),H_{O}^{T}\right\} }\le0.
\]
\end{itemize}
\end{enumerate}
A function $U=\left(U_{c},U_{m+1},\ldots,U_{N}\right)$
where $U_{c}\in LSC\left(\cup_{j=1}^{m}\Gamma_{j};\mathbb{R}\right),U_{i}\in LSC\left(\Gamma_{i};\mathbb{R}\right)$
is called a \emph{viscosity super-solution} of the Hamilton-Jacobi
system if 
\begin{equation}
  \label{eq:3}
U_{c}\left(O\right)\ge U_{i}\left(O\right), \quad \hbox{for } i=\overline{m+1,N},
\end{equation}
and for any $\left(\varphi_{c},\varphi_{m+1},\ldots,\varphi_{N}\right)\in R\left(\mathcal{G}\right)$:
\begin{enumerate}
\item  
if  $U_{c}-\varphi_{c}$ has a local maximum  at $x_{c}\in \cup_{j=1}^{m}\Gamma_{j}$ and if
\begin{itemize}
\item  $x_{c}\in \Gamma_{j}\backslash\left\{ O\right\} $ for some $j\le m$, then
\[
\begin{array}{c}
\lambda U_{c}\left(x_{c}\right)+H_{j}\left(x,\dfrac{d\varphi_{c}}{d x_{j}}\left(x_{c}\right)\right)\ge0,
\end{array}
\]
\item  $x_{c}=O$, then
\[
\lambda U_{c}\left(O\right)+\max\left\{ -\lambda{\displaystyle \min_{j>m}}\left\{ U_{j}\left(O\right)+\overline{c}_{j}\right\} ,\max_{j\le m}\left\{ H_{j}^{+}\left(O,\dfrac{d\varphi_{c}}{d x_{j}^{+}}\left(O\right)\right)\right\} ,H_{O}^{T}\right\} \ge0;
\]
\end{itemize}
\item if $U_{i}-\varphi_{i}$ has a \emph{local minimum point}
 at $x_{i}\in \Gamma_{i}$ for $i>m$, and 
if
\begin{itemize}
\item $\ensuremath{x_{i}\in \Gamma_{i}\backslash \left\{O\right\}}$, then
\[
\begin{array}{c}
\lambda U_{i}\left(x_{i}\right)+H_{i}\left(x,\dfrac{d\varphi_{i}}{d x_{i}}\left(x_{i}\right)\right)\ge0,\end{array}
\]
\item  $x_{i}=O$ for $i>m$ then
\[
{\displaystyle \lambda U_{i}\left(O\right)+\max\left\{ -\lambda\min_{j>m,j\ne i}\left\{ U_{j}\left(O\right)+\overline{c}_{j}\right\} ,-\lambda U_{c}\left(O\right),H_{i}^{+}\left(O,\dfrac{d\varphi_{i}}{dx_{i}}\left(O\right)\right),H_{O}^{T}\right\} }\ge0.
\]
\end{itemize}
\end{enumerate}
A function $U=\left(U_{c},U_{1},\ldots,U_{m}\right)$
where $U_{c}\in C\left(\cup_{j\le m}\Gamma_{j};\mathbb{R}\right)$ and
$U_{i}\in C\left(\Gamma_{i};\mathbb{R}\right)$ for all $i>m$ is called
a \emph{viscosity solution} of the Hamilton-Jacobi system if
it is both a viscosity sub-solution and a viscosity super-solution
of the Hamilton-Jacobi system.
\end{defn}

\begin{rem}
The term $-\lambda H_{C}\left(O\right)$ in the above definition accounts for the situation in which the trajectory enters $\cup_{j=1}^{m}\Gamma{j}$. The term $\max_{j\le m}\left\{ H_{j}^{+}\left(O,\dfrac{d\varphi_{c}}{d x_{j}^{+}}\left(O\right)\right)\right\}$ accounts for the situation in which the trajectory enters $\Gamma_{i_0}$ where $ H_{i_0}^{+}\left(O,\dfrac{d\varphi_{c}}{d x_{j}^{+}}\left(O\right)\right)=\max_{j\le m}\left\{ H_{j}^{+}\left(O,\dfrac{d\varphi_{c}}{d x_{j}^{+}}\left(O\right)\right)\right\}$.
\end{rem}

\begin{rem}
In the case when $\overline{c}_{i}=0$ for $i=\overline{1,N}$, i,e., $m=N$, the term $-\lambda \min_{j>m}U_j \left( O \right)+\overline{c}_j$ vanishes. This implies that
\begin{align*}
\max\left\{ -\lambda{\displaystyle \min_{j>m}}\left\{ U_{j}\left(O\right)+\overline{c}_{j}\right\} ,\max_{j\le m}\left\{ H_{j}^{+}\left(O,\dfrac{\partial\varphi_{c}}{\partial e_{j}^{+}}\left(O\right)\right)\right\} ,H_{O}^{T}\right\} = & \max_{j=\overline{1,N}}\left\{ H_{j}^{+}\left(O,\dfrac{\partial\varphi_{c}}{\partial e_{j}^{+}}\left(O\right)\right)\right\} \\
= & H_{O}\left(\dfrac{\partial\varphi_{c}}{\partial e_{1}^{+}}\left(O\right),\ldots,\dfrac{\partial\varphi_{c}}{\partial e_{N}^{+}}\left(O\right)\right).
\end{align*}
where $H_O \left( p_1,\ldots,p_N \right)$ is defined in \cite[page 6]{AOT2015}. This means that, in the case when all the entry costs $\overline{c}_j$ vanish, we recover the notion of viscosity solution proposed  in~\cite{AOT2015}.
\end{rem}

We now study the relationship between the value function $\mathsf{V}$ and the Hamilton-Jacobi system.

\begin{thm}\label{sec:more-general-optimal}
Let $\mathsf{V}$ be the value function corresponding to the entry costs $\overline{C}$, then 
$\left(\mathcal{V}_c,\mathcal{V}_{m+1},\ldots,\mathcal{V}_N\right)$ is a
viscosity solution of the Hamilton-Jacobi system.
\end{thm}
Let us  state the comparison principle for the  Hamilton-Jacobi system.
\begin{thm}\label{sec:more-general-optimal-1}
Let $U=\left(U_{c},U_{m+1},\ldots,U_{N}\right)$ and $ W=\left(W_{c},W_{m+1},\ldots,W_{N}\right)$ be a bounded
viscosity sub-solution and a viscosity super-solution, respectively,  of the Hamilton-Jacobi system.
The following holds:
 $U\le W$ in $\mathcal{G}$, i.e., $U_{c}\le W_{c}$
on $\cup_{j=1}^{m}\Gamma_{j}$, and $U_{i}\le W_{i}$ in $\Gamma_{i}$ for all $i>m$.
\end{thm}
\begin{proof}[Proof of Theorem~\ref{sec:more-general-optimal-1}]
Suppose by contradiction that there exists $i\in\left\{ 1,\ldots,N\right\} $ and $x\in \Gamma_{i}$
such that
\[
\begin{cases}
U_{c}\left(x\right)-W_{c}\left(x\right)>0 & \text{if }i\le m,\\
U_{i}\left(x\right)-W_{i}\left(x\right)>0 & \text{if }i>m,
\end{cases}
\]
then
\[
\begin{cases}
U_{c}\left(O\right)-W_{c}\left(O\right)=\max_{\cup_{j=1}^{m}\Gamma_{j}}\left\{ U_{c}-W_{c}\right\} >0 & \text{if }i\le m,\\
U_{i}\left(O\right)-W_{i}\left(O\right)=\max_{\Gamma_{i}}\left\{ U_{i}-W_{i}\right\} >0 & \text{if }i>m,
\end{cases}
\]
since the case where the positive maximum is achieved outside the junction leads to a contradition by classical comparison results.

\begin{description}
  \item{\emph{Case 1:}} ${\displaystyle U_{c}\left(O\right)-W_{c}\left(O\right)=\max_{\cup_{i=1}^{m}\Gamma_{i}}\left(U_{c}-W_{c}\right)}>0$
    \begin{description}
      \item{\emph{Sub-case 1-a:}}  $W_{c}\left(O\right)<\min_{j>m}\left\{ W_{j}\left(O\right)+\overline{c}_{j}\right\} $.
Since $W_{c}\left(O\right)<U_{c}\left(O\right)\le-\dfrac{H_{O}^{T}}{\lambda}$,
the function $W_{c}$ is a viscosity super-solution of
\[
\begin{cases}
\lambda W_{c}\left(x\right)+H_{i}\left(x,\dfrac{d W_{c}}{d x_{i}}\left(x\right)\right) & = 0\quad\mbox{if }i\le m,x\in \Gamma_{i}\backslash\left\{ O\right\} ,\\
\lambda W_{c}\left(O\right)+H_{c}\left(\dfrac{d W_{c}}{d x_{1}^{+}}\left(O\right),\ldots,\dfrac{d W_{c}}{d x_{m}^{+}}\left(O\right)\right) & = 0\quad\mbox{if }x=O.
\end{cases}
\]
where $H_{c}\left(p_{1},\ldots,p_{m}\right)=\max_{i\le m}H_{i}^{+}\left(O,p_{i}\right)$.
Applying Lemma~\ref{lem: Comparison Principle} in the Appendix, we obtain that 
$U_{c}\left(O\right)\le W_{c}\left(O\right)$ in contradiction with the assumption.
\item{\emph{Sub-case 1-b:}} $W_{c}\left(O\right)\ge\min_{j>m}\left\{ W_{j}\left(O\right)+\overline{c}_{j}\right\} =W_{i_{0}}\left(O\right)+\overline{c}_{i_{0}}$.
Since $\overline{c}_{i_{0}}>0$, we first see that $W_{i_{0}}\left(O\right)<\min\left\{ \min_{j>m}\left\{ W_{j}\left(O\right)+\overline{c}_{j}\right\} ,W_{c}\left(O\right),-\dfrac{H_{O}^{T}}{\lambda}\right\} $.
Hence, $W_{i_{0}}$ is a viscosity super-solution of~\eqref{eq:H-J equation with H+} replacing $i$ by $i_{0}$. Moreover, since
\[
U_{i_{0}}\left(O\right)+\overline{c}_{i_{0}}\ge\min_{j>m}\left(U_{j}\left(O\right)+\overline{c}_{j}\right)\ge U_{c}\left(O\right)>
W_{c}\left(O\right)>W_{i_0}\left(O\right)+\overline{c}_{i_0},
\]
then $U_{i_{0}}\left(O\right)>W_{i_{0}}\left(O\right)$. Applying the same argument as \emph{Case
1} in the second proof of
Theorem~\ref{thm: Comparison Principle} replacing $i$ by $i_0$,  we obtain that $U_{i_{0}}\left(O\right)\le W_{i_{0}}\left(O\right)$,
which is contradictory.
    \end{description}
\item{\emph{Case 2:}} ${\displaystyle U_{i}\left(O\right)-W_{i}\left(O\right)=\max_{\Gamma_{i}}\left(U_{i}-W_{i}\right)}>0$
for some $i>m$. Using the definition of viscosity sub-solutions and \emph{Case 1},
 we see that $W_{i}\left(O\right)<U_{i}\left(O\right)\le U_{c}\left(O\right)\le W_{c}\left(O\right)$. 
\begin{description}
\item{\emph{Sub-case 2-a:}} $W_{i}\left(O\right)<\min_{j>m}\left\{ W_{j}\left(O\right)+\overline{c}_{j}\right\} $.
Since $U_{i}\left(O\right)<-\dfrac{H_{O}^{T}}{\lambda}$, we first
see that $W_{i}\left(O\right)<\min\left\{ \min_{j>m}\left\{ W_{j}\left(O\right)+\overline{c}_{j}\right\} ,W_{c}\left(O\right),-\dfrac{H_{O}^{T}}{\lambda}\right\} $.
Hence, $W_{i}$ is a viscosity super-solution of~\eqref{eq:H-J equation with H+}. Applying the same argument as in \emph{Case 1} in the second proof of
Theorem~\ref{thm: Comparison Principle}, we see that $U_{i}\left(O\right)\le W_{i}\left(O\right)$, which is  contradictory.
\item{\emph{Sub-case 2-b:}} $W_{i}\left(O\right)\ge\min_{j>m}\left\{ W_{j}\left(O\right)+\overline{c}_{j}\right\} =W_{i_{0}}\left(O\right)+\overline{c}_{i_{0}}$.
Since $\overline{c}_{i_{0}}>0$, we can check that $W_{i_{0}}\left(O\right)<\min\left\{ \min_{j>m}\left\{ W_{j}\left(O\right)+\overline{c}_{j}\right\} ,W_{c}\left(O\right),-\dfrac{H_{O}^{T}}{\lambda}\right\} $.
Hence, $W_{i_{0}}$ is a viscosity super-solution of~\eqref{eq:H-J equation with H+} replacing $i$ by $i_{0}$. Moreover, since
\[
U_{i_{0}}\left(O\right)+\overline{c}_{i_{0}}\ge\min_{j>m}\left(U_{j}\left(O\right)+\overline{c}_{j}\right)\ge U_{c}\left(O\right)
> W_{i}\left(O\right)>W_{i_{0}}\left(O\right)+\overline{c}_{i_{0}},
\]
then $U_{i_{0}}\left(O\right)>W_{i_{0}}\left(O\right)$. Applying the same argument as \emph{Case
1} in the second proof of
Theorem~\ref{thm: Comparison Principle} replacing $i$ by $i_0$, we obtain that $U_{i_{0}}\left(O\right)\le W_{i_{0}}\left(O\right)$
which is contradictory.
\end{description}
\end{description}
\end{proof}
\appendix
\section{Appendix}
\label{sec:gener-case-switch-1}
\begin{lem}
\label{lem: Apendix 1} For any $a\in A_{i}^{+}$, there exists a sequence
$\left\{ a_{n}\right\}$ such that $a_{n}\in A_{i}$
and
\begin{eqnarray*}
f_{i}\left(O,a_{n}\right) & \ge & \dfrac{\delta}{n}>0,\\
\left|f_{i}\left(O,a_{n}\right)-f_{i}\left(O,a\right)\right| & \le & \dfrac{2M}{n},\\
\left|\ell_{i}\left(O,a_{n}\right)-\ell_{i}\left(O,a\right)\right| & \le & \dfrac{2M}{n}.
\end{eqnarray*}
\end{lem}
\begin{proof}[Proof of Lemma~\ref{lem: Apendix 1}]
From assumption $\left[H4\right]$, there exists $a_{\delta}\in A_{i}$
such that $f_{i}\left(O,a_{\delta}\right)=\delta$. Since $\mbox{FL}_{i}\left(O\right)$
is convex (by assumption $\left[H3\right]$), for any $n\in\mathbb{N},a\in A^{+}_{i}$
\[
\dfrac{1}{n}\left(f_{i}\left(O,a_{\delta}\right)e_{i},\ell_{i}\left(O,a_{\delta}\right)\right)+\left(1-\dfrac{1}{n}\right)\left(f_{i}\left(O,a\right),\ell_{i}\left(O,a\right)e_{i}\right)\in\mbox{FL}_{i}\left(O\right).
\]
Then, there exists a sequence $\left\{ a_{n}\right\}$
such that $a_{n}\in A_{i}$ and
\begin{equation}
\dfrac{1}{n}\left(f_{i}\left(O,a_{\delta}\right),\ell_{i}\left(O,a_{\delta}\right)\right)+\left(1-\dfrac{1}{n}\right)\left(f_{i}\left(O,a\right),\ell_{i}\left(O,a\right)\right)=\left(f_{i}\left(O,a_{n}\right),\ell_{i}\left(O,a_{n}\right)\right)\in\mbox{FL}_{i}\left(O\right).\label{eq: convex property}
\end{equation}
Notice that $f_{i}\left(O,a\right)\ge0$ since $a\in A_{i}^{+}$, this yields
\[
f_{i}\left(O,a_{n}\right)\ge\dfrac{f_{i}\left(O,a_{\delta}\right)}{n}=\dfrac{\delta}{n}>0.
\]
From~\eqref{eq: convex property}, we also have
\[
\left|f_{i}\left(O,a_{n}\right)-f_{i}\left(O,a\right)\right|=\dfrac{1}{n}\left|f_{i}\left(O,a_{\delta}\right)-f_{i}\left(O,a\right)\right|\le\dfrac{2M}{n},
\]
and
\[
\left|\ell_{i}\left(O,a_{n}\right)-\ell_{i}\left(O,a\right)\right|=\dfrac{1}{n}\left|\ell_{i}\left(O,a_{\delta}\right)-\ell_{i}\left(O,a\right)\right|\le\dfrac{2M}{n}.
\]
\end{proof}
We can state the following corollary of Lemma~\ref{lem: Apendix 1}:
\begin{cor}
\label{corollary: max=00003Dsup}For  $i=\overline{1,N} $
and $p_{i}\in\mathbb{R}$,
\[
\max_{a\in A_{i}\mbox{ s.t. }f_{i}\left(O,a\right)\ge0}\left\{ -f_{i}\left(O,a\right)p_{i}-\ell_{i}\left(O,a\right)\right\} =\sup_{a\in A_{i}\mbox{ s.t. }f_{i}\left(O,a\right)>0}\left\{ -f_{i}\left(O,a\right)p_{i}-\ell_{i}\left(O,a\right)\right\} .
\]
\end{cor}
\begin{lem}
\label{lem: Comparison Principle}If $U_{c}$ and $W_{c}$ are respectively
viscosity sub and super-solution of
\begin{align*}
\lambda U_{c}\left(x\right)+H_{i}\left(x,\dfrac{d U_{c}}{d x_{i}}\left(x\right)\right) & \le0\text{ if }x\in \Gamma_{i}\backslash\left\{ O\right\} ,\\
\lambda U_{c}\left(O\right)+H_{c}\left(\dfrac{d U_{c}}{d x_{1}}\left(O\right),\ldots,\dfrac{d U_{c}}{d x_{m}}\left(O\right)\right) & \le0\text{ if }x=O,
\end{align*}
and 
\begin{align*}
\lambda W_{c}\left(x\right)+H_{i}\left(x,\dfrac{d W_{c}}{d x_{i}}\left(x\right)\right) & \ge0\text{ if }x\in \Gamma_{i}\backslash\left\{ O\right\} ,\\
\lambda W_{c}\left(O\right)+H_{c}\left(\dfrac{d W_{c}}{d x_{1}}\left(O\right),\ldots,\dfrac{d W_{c}}{d x_{m}}\left(O\right)\right) & \ge0\text{ if }x=O,
\end{align*}
then $U_{c}\left(x\right)\le W_{c}\left(x\right)$ for all $x\in\bigcup_{i=1}^{m}\Gamma_{i}$.
\end{lem}
\begin{proof}[Proof of Lemma~\ref{lem: Comparison Principle}]
Assume that there exists $\widehat{x}\in \Gamma_{i}$ where $1\le i\le m$
and $U_{c}\left(\widehat{x}\right)-W_{c}\left(\widehat{x}\right)>0$. By classical
comparison principle for the boundary problem on $\Gamma_{i}$, one gets 
\[
U_{c}\left(O\right)-W_{c}\left(O\right)=\max_{\Gamma_{i}}\left\{ U_{c}\left(x\right)-W_{c}\left(x\right)\right\} >0.
\]
Applying again classical comparison principle for the boundary problem for each edge $\Gamma_{j}$
\[
U_{c}\left(O\right)-W_{c}\left(O\right)=\max_{\bigcup_{i=1}^{m}\Gamma_{i}}\left\{ U_{c}\left(x\right)-W_{c}\left(x\right)\right\} >0.
\]
For $j=\overline{1,N}$, we consider the function
\begin{align*}
\Psi_{j,\varepsilon,\gamma}:\Gamma_{j}\times \Gamma_{j} & \longrightarrow\mathbb{R}\\
\left(x,y\right) & \longrightarrow U_{c}\left(x\right)-W_{c}\left(y\right)-\dfrac{1}{2\varepsilon}\left[-\left|x\right|+\left|y\right|+\delta\left(\varepsilon\right)\right]^{2}-\gamma\left(\left|x\right|+\left|y\right|\right),
\end{align*}
where $\delta\left(\varepsilon\right)=\left(L+1\right)\varepsilon$,
$\gamma\in\left(0,\dfrac{1}{2}\right)$.\\
The function $\Psi_{j,\varepsilon}$
attains its maximum at $\left(x_{j,\varepsilon,\gamma},y_{j,\varepsilon,\gamma}\right)\in \Gamma_{j}\times \Gamma_{j}$.
Applying the same argument as in the second proof of
Theorem~\ref{thm: Comparison Principle}, we have $x_{j,\varepsilon,\gamma},y_{j,\varepsilon,\gamma}\rightarrow O$
and $\dfrac{\left(x_{j,\varepsilon,\gamma}-y_{j,\varepsilon,\gamma}\right)^{2}}{\varepsilon}\rightarrow0$
as $\varepsilon\rightarrow0$. Moreover, for
any $j=\overline{1,m}$, $x_{j,\varepsilon,\gamma}\ne O$. 
We claim that $y_{j,\varepsilon,\gamma}$ must be $O$ for $\varepsilon$ small enough . Indeed, if there exists a sequence $\varepsilon_{n}$ such that $y_{j,\varepsilon_{n},\gamma}\in \Gamma_{j}\backslash \left\{O \right\} $, then applying viscosity inequalities, we have
\begin{align*}
U_{c}\left(x_{j,\varepsilon_{n},\gamma}\right)+H_{j}\left(x_{j,\varepsilon_{n},\gamma},\dfrac{-x_{j,\varepsilon_{n},\gamma}+y_{j,\varepsilon_{n},\gamma}+\delta\left(\varepsilon_{n}\right)}{\varepsilon_{n}}+\gamma\right) & \le0,\\
W_{c}\left(y_{j,\varepsilon_{n},\gamma}\right)+H_{j}\left(y_{j,\varepsilon_{n},\gamma},\dfrac{-x_{j,\varepsilon_{n},\gamma}+y_{j,\varepsilon_{n},\gamma}+\delta\left(\varepsilon_{n}\right)}{\varepsilon_{n}}-\gamma\right) & \ge0.
\end{align*}
Subtracting the two inequalities and using~\eqref{ineq: estimate Hamiltonian} with $H_{j}$, we obtain
\[
U_{c}\left(x_{j,\varepsilon_{n},\gamma}\right)-W_{c}\left(y_{j,\varepsilon_{n},\gamma}\right)\le\overline{M}_{j}\left|x_{j,\varepsilon_{n},\gamma}-y_{j,\varepsilon_{n},\gamma}\right|\left(1+\left|\dfrac{-x_{j,\varepsilon_{n},\gamma}+y_{j,\varepsilon_{n},\gamma}+\delta\left(\varepsilon_{n}\right)}{\varepsilon_{n}}-\gamma\right|\right)+\overline{M}_{j}2\gamma.
\]
Recall that we already have $\dfrac{\left(x_{j,\varepsilon_{n},\gamma}-y_{j,\varepsilon_{n},\gamma}\right)^{2}}{\varepsilon_{n}}\rightarrow0$
as $n\rightarrow{\infty}$. Let $n$ tend to $\infty$ and $\gamma$ tend to $0$ then we
obtain $U_{c}\left(O\right)-W_{c}\left(O\right)\le0$. It leads us
to a contradiction. So this claim is proved.

Define the function  $\Psi:\bigcup_{j=1}^{m}\Gamma_{j}\rightarrow\mathbb{R}$ by
\[
\Psi|_{\Gamma_{i}}\left(y\right)=\dfrac{1}{2\varepsilon}\sum_{j\ne i}\left\{ \left[-\left|x_{i,\varepsilon,\gamma}\right|+\delta\left(\varepsilon\right)\right]^{2}-\gamma\left|x_{i,\varepsilon,\gamma}\right|\right\} +\dfrac{1}{2\varepsilon}\left[-\left|x_{i,\varepsilon,\gamma}\right|+\left|y\right|+\delta\left(\varepsilon\right)\right]^{2}+\gamma\left(-\left|x_{i,\varepsilon,\gamma}\right|+\left|y\right|\right).
\]
We can see that $\Psi$ is continuous on $\bigcup_{j=1}^{m}\Gamma_{j}$ and belongs to $C^{1}\left(\Gamma_{j}\right)$ for $j=\overline{1,m}$. Moreover, for $j=\overline{1,m}$ and for $\varepsilon$ small enough, $y_{j,\varepsilon,\gamma}$=O  then the function $\Psi$ + $W_{c}$ has a minimum point at $O$. It yields
\[
\lambda W_{c}\left(O\right)+H_{c}\left(\dfrac{-\overline{x}_{1,\varepsilon,\gamma}+\delta\left(\varepsilon\right)}{\varepsilon},\ldots,\dfrac{-\overline{x}_{m,\varepsilon,\gamma}+\delta\left(\varepsilon\right)}{\varepsilon}\right)\ge0.
\]
By definition of $H_{c}$, there exists $j_{0}\in\left\{ 1,\ldots,m\right\} $
such that
\[
\lambda W_{c}\left(O\right)+H_{j_{0}}^{+}\left(O,\dfrac{-\overline{x}_{j_{0},\varepsilon,\gamma}+\delta\left(\varepsilon\right)}{\varepsilon}\right)\ge0.
\]
This implies
\[
\lambda W_{c}\left(O\right)+H_{j_{0}}\left(O,\dfrac{-\overline{x}_{j_{0},\varepsilon,\gamma}+\delta\left(\varepsilon\right)}{\varepsilon}\right)\ge0
\]
On the other hand, since $x_{j_{0},\varepsilon,\gamma}\in \Gamma_{j_0}\backslash\left\{ O\right\} $,
we have 
\[
\lambda U_{c}\left(\overline{x}_{j_0,\varepsilon,\gamma}\right)+H_{j_{0}}\left(x_{j_0,\varepsilon,\gamma},\dfrac{-\overline{x}_{j_0,\varepsilon,\gamma}+\delta\left(\varepsilon\right)}{\varepsilon}\right)\le0.
\]
Subtracting the two inequalities and using properties of Hamiltonian $H_{j_{0}}$,
let $\varepsilon$ tend to $0$ then $\gamma$ tend to $0$, we obtain that
$U_{c}\left(O\right)-W_{c}\left(O\right)\le0$, which is contradictory.
\end{proof}
\section*{Acknowledgement}
I would like to express my thanks to my advisors Y. Achdou, O. Ley and N. Tchou for suggesting me this work and for their help. I also thank the hospitality of Centre Henri Lebesgue and INSA de Rennes during
the preparation of this work. Moreover, this work was partially supported by the ANR (Agence Nationale de la Recherche)
through HJnet project ANR-12-BS01-0008-01 and MFG project ANR-16-CE40-0015-01.

\end{document}